\documentclass[12pt,a4paper,english,reqno]{amsart}
\usepackage{amsmath,amsfonts,amssymb,mathrsfs}
\usepackage[english]{babel}
\usepackage[utf8x]{inputenc}
\usepackage[T1]{fontenc}
\usepackage{graphicx}
\usepackage{tabularx}
\usepackage{hyperref}
\usepackage{mathtools}
\usepackage{cleveref}
\usepackage{comment}
\usepackage{xypic}
\numberwithin{equation}{subsection}

\newcommand{\R}{\mathbb{R}}

\renewcommand{\H}{\mathbb{H}}
\newcommand{\Z}{\mathbb{Z}}
\newcommand{\Q}{\mathbb{Q}}

\newcommand{\N}{\mathbb{N}}

\newcommand{\C}{\mathbb{C}}

\newcommand{\sgn}{\mathrm{sgn}}

\newcommand{\Pic}{\mathrm{Pic}}
\renewcommand{\P}{\mathbb P}

\renewcommand{\cD}{\mathbb D}
\newcommand{\twopartdef}[4]
{
	\left\{
		\begin{array}{ll}
			#1 & \mbox{if } #2 \\
			#3 & \mbox{if } #4
		\end{array}
	\right.
}

\theoremstyle{plain}

\newtheorem{theorem}{Theorem}[section]

\newtheorem{proposition}[theorem]{Proposition}

\newtheorem{lemma}[theorem]{Lemma}

\newtheorem{corollaire}[theorem]{Corollary}
\newtheorem{corollary}[theorem]{Corollary}

\theoremstyle{definition}
\newtheorem{definition}[theorem]{Definition}

\theoremstyle{remark}
\newtheorem{remarque}[theorem]{Remark}

\newtheorem{notation}[theorem]{Notation}
\usepackage[usenames,dvipsnames]{color} 

\newcommand\numberthis{\addtocounter{equation}{1}\tag{\theequation}}

\title{Mixed mock modularity of special divisors}
\author{Philip Engel}
\address{Department of Mathematics, University of Georgia, Boyd Hall, Athens, GA 30602}
\email{philip.engel@uga.edu}
\author{Fran\c{c}ois Greer}
\address{Department of Mathematics, Michigan State University, 619 Red Cedar Rd, East Lansing, MI 48824}
\email{greerfra@msu.edu}
\author{Salim Tayou}
\address{Department of Mathematics, Harvard University, 1 Oxford St, Cambridge, MA 02138, USA}
\email{tayou@math.harvard.edu}
\date{\today}

\begin{document}

\begin{abstract}
We prove that the generating series of special divisors in toroidal compactifications of orthogonal Shimura varieties is a mixed mock modular form. More precisely, we find an explicit completion using theta series associated to rays in the cone decomposition. The proof relies on intersection theory at the boundary of the Shimura variety.  

\end{abstract}

\maketitle
\setcounter{tocdepth}{1}
\tableofcontents

\section{Introduction} 
Let $(L,\cdot)$ be a non-degenerate even lattice of signature $(2,n)$ with quadratic form $Q(x)=\tfrac{1}{2}x\cdot x$. Associated to $L$ is a type IV Hermitian symmetric domain $\cD$, defined as a connected component of
$$\{ \C x \in \P(L_\C)\,\big{|}\, x\cdot x=0, \, x\cdot \bar x>0\}.$$ 
It is a complex manifold of dimension $n$. Let $O^+(L)$ denote the group of orthogonal transformations of $L$  which preserves the component $\cD$, and let $\Gamma\subset O^+(L)$ be a finite index subgroup which acts trivially  on the discriminant group $L^{\vee}/L$, where $L^{\vee}$ is the dual lattice. 

The arithmetic quotient $X_\Gamma:=\Gamma\backslash \cD$ is a connected complex Shimura variety of orthogonal type. It parameterizes polarized Hodge structures on the lattice $L$ with Hodge numbers $(1,n,1)$. Such Hodge structures arise naturally in the study of polarized K3 surfaces when $n=19$, and additionally of elliptic curves, abelian surfaces with real multiplication, and abelian surfaces, when $n=1,2,3$, respectively. 

There is a countable family of divisors, called \emph{special divisors}\footnote{or Heegner divisors}, which are the images of orthogonal Shimura varieties associated to sublattices of $L$ of signature $(2,n-1)$. Geometrically, they correspond to Noether--Lefschetz loci for the polarized variation of Hodge structure parameterized by $X_\Gamma$. For every $\beta\in L^{\vee}/L$ and $m\in -Q(\beta)+\Z$, there is an associated special divisor $Z^o(\beta,m)\subset X_\Gamma$ defined in Equation \ref{special-divisor}. We denote its class in the cohomology group $H^2(X_\Gamma,\Q)$ by $[Z^o(\beta,m)]$. Define $[Z^o(0,0)]$ to be the first Chern class of the dual of the Hodge bundle on $X_\Gamma$. The special cycles can be packaged in the following generating series with values in $H^2(X_\Gamma,\Q)\otimes \C[L^\vee/L]$: 

\begin{align}\label{generating-series}
[Z^o(0,0)]\otimes \mathfrak{e}_0 +\sum_{\beta\in L^{\vee}/L}\sum_{m\in -Q(\beta)+\Z} [Z^o(\beta,m)]\otimes \mathfrak{e}_\beta q^{m}.
\end{align}

In an influential paper \cite{borcherdszagier}, Borcherds proved that Equation \ref{generating-series} is a vector-valued modular form of weight $1+\frac{n}{2}$ and representation $\overline{\rho}_L$ valued in ${\rm Pic}(X_\Gamma)$, where $\rho_L$ denotes the Weil representation associated to the discriminant lattice $(L^{\vee}/L,Q)$. His proof relied on earlier work \cite{borcherds-inventiones-automorphic}, which constructed meromorphic modular forms on $X_\Gamma$ whose zeroes and poles are special divisors. A cohomological version was known previously by the work of Kudla--Millson \cite{kudla-millson}, see also \cite{garcia}, and Hirzebruch-Zagier in the case of Hilbert modular varieties \cite{hirzag}.

This paper is concerned with the modularity behavior of Equation \ref{generating-series}, in compactifications of the variety $X_\Gamma$. In general, $X_\Gamma$ is non-compact and admits several compactifications. The first compactification was constructed by Satake \cite{satake} and Baily-Borel \cite{bailyborel} and is called the \emph{Baily-Borel compactification}. It is typically singular at the boundary, which has codimension at least $n-1$.

We will be interested in the {\it toroidal compactifications} $X_\Gamma\hookrightarrow X_\Gamma^{\Sigma}$, first constructed in \cite{mumford-amrt}. They always have divisorial boundary, but are non-unique, depending on certain combinatorial data $\Sigma$. This data consists, for each isotropic line $I\subset L$, of a polyhedral cone decomposition $\Sigma_I$ of the positive cone of the hyperbolic lattice $I^\perp/I$.
These decompositions $\Sigma_I$ must additionally be $\Gamma$-invariant. We refer to \Cref{s:tor-comp} for more details.

We will assume that the chosen polyhedral cone decomposition is simplicial, so that $X_\Gamma^\Sigma$ has finite quotient singularities. It is also possible to choose $\Gamma$ and $\Sigma$ so that $X_\Gamma^{\Sigma}$ is a smooth projective algebraic variety, although this is not necessary for our main results.


\begin{theorem} \label{t:main}
The generating series \begin{align}\label{generating-series-closure}
[Z(0,0)]\otimes \mathfrak{e}_0 +\sum_{\beta\in L^{\vee}/L}\sum_{m\in -Q(\beta)+\Z} [Z(\beta,m)]\otimes \mathfrak{e}_\beta q^{m}
\end{align} of cohomology classes of Zariski closures $Z(\beta,m):= \overline{Z^o(\beta,m)}$, valued in $H^2(X_\Gamma^\Sigma,\Q)\otimes \C[L^\vee/L]$, is a mixed mock modular form of weight $1+n/2$ with respect to $\overline{\rho}_L$. 
\end{theorem}

Mixed mock modular forms are holomorphic parts of mixed harmonic Maass forms which have been studied by \cite{bruinierfunke,folsom} and which we define in \Cref{s:harmonic-mass-forms}. 

To give a more precise form of the above theorem, we introduce some notation: for any isotropic plane $J\subset L$, the lattice $M:=J^{\bot}/J$ is negative definite and has an associated theta series $\Theta_M$ of weight $\frac{n}{2}-1$ and representation $\overline{\rho}_M$. The image of this theta series under an operator between Weil representations $p_M^L$ is a modular form of the same weight, but for the representation $\overline{\rho}_L$.

On the other hand, for any isotropic line $I\subset L$ and a ray $\R_{\geq 0}c$ in the cone decomposition $\Sigma_I$ with $c\cdot c=2N>0$, the lattice $c^{\bot}\subset I^\perp/I$ is negative definite, and we can similarly consider the theta series $\Theta_{c^{\bot}}$ which is a vector-valued holomorphic modular form of weight $\frac{n-1}{2}$ and representation $\overline{\rho}_{c^{\bot}}$. The theta series $\Theta_N$ of the one dimensional lattice $\Z c$ also plays a role: there exists a weak harmonic Maass form $F_N$ of weight $3/2$ and representation $\overline{\rho}_{\Z c}$, originally constructed by Zagier for $N=1$ and generalized by Funke and Bruinier-Schwagenscheidt, whose shadow is proportional to $\Theta_N$, see \Cref{s:zagier}.
Our main theorem has the following more precise form. 
\begin{theorem}\label{t:precise-main}
The generating series 
\begin{align*} 
\sum_{\beta\in L^{\vee}/L}\,\sum_{m\in Q(\beta)+\Z} &[Z(\beta,m)]\otimes \mathfrak{e}_\beta q^{m} -\frac{1}{2}\sum_{(I,c)} \uparrow_{c^\perp \oplus \Z c}^{L}\left(\Theta_{c^\perp}\otimes F^+_N\right)\otimes \Delta_{I,c} \\ &+ \frac{1}{12}\sum_{J}
E_2(q)\cdot \uparrow_{M}^{L}\left(\Theta_M\right)\otimes \Delta_J
\end{align*}
valued in $H^2(X_\Gamma^\Sigma,\Q)\otimes \C[L^\vee/L]$ is a weakly holomorphic modular form of weight $1+\frac{n}{2}$ and representation $\overline{\rho}_L$. 
\end{theorem}

The quasi-modular form $E_2$ in the Theorem above is defined as $E_2(\tau)=1-24\sum_{n\geq 1}\sigma_1(n)q^n$. Also, $\Delta_J$ and $\Delta_{I,c}\subset X_\Gamma^\Sigma$ are the boundary divisors of the toroidal compactification $X_\Gamma^\Sigma$, associated to ($\Gamma$-orbits of) isotropic planes $J\subset L$ and rays $\R_{\geq 0}c\in \Sigma_I$. 
As $\uparrow_{M}^{L}\left(\Theta_M\right)$ and $\uparrow_{c^\perp\oplus \Z c}^{L}\left(\Theta_{c^\perp}\otimes F^+_N\right)$ are mixed mock modular forms, \Cref{t:precise-main} implies \Cref{t:main}. 

Using Margulis' super-rigidity theorem, we prove that the first Betti number of $X_\Gamma^{\Sigma}$ vanishes when $n\geq 3$ or if $n=2$ and $\Gamma$ is an irreducible lattice, see \Cref{s:superR}. We get the following corollary. 
\begin{corollaire}\label{c:modularity-picard}
Assume either $n\geq 3$ or the Witt index of $L$ is $1$ and $n=2$. Then the generating series in Equation \ref{generating-series-closure} is also a mixed mock modular form in $\mathrm{Pic}(X_\Gamma^{\Sigma})_\Q$ of weight $1+\frac{n}{2}$ and representation $\overline{\rho}_L$. 
\end{corollaire}

\subsection{Strategy of the proof}
We prove \Cref{t:precise-main} by showing that the pairing of the generating series \ref{generating-series-closure} with every homology class in $H_2(X_\Gamma^\Sigma,\Q)$ is a holomorphic modular form of weight $1+\frac{n}{2}$ and level $\overline{\rho}_L$. For this, we prove a splitting theorem stating that the second homology of $X_\Gamma^\Sigma$ is rationally generated by the homology of the open part $X_\Gamma$ and the classes of algebraic curves $C\subset X^\Sigma_\Gamma$ contracted in the Baily-Borel compactification. For homology classes supported in the open part, we use Kudla--Millson's result \cite{kudla-millson} as an input to deduce that the generating series is modular.

Our second contribution is to compute the intersection of the special divisors with the curves $C\subset X_\Gamma^\Sigma$ in the boundary. The classes of such curves are generated by the one-dimensional toric boundary strata of $X_\Gamma^\Sigma$. We use tools from toric geometry to explicitly compute such intersection numbers. 

The resulting generating series is a theta series of a real hyperbolic lattice, with a piecewise linear weighting determined by the cone decomposition $\Sigma_I$. We find suitable non-holomorphic completions \`a la Zwegers, and prove that they are non-holomorphic modular forms by a theorem of Vigneras \cite{vigneras}. Finally, we determine the theta series explicitly, in terms of the rays in the cone decomposition. In some sense, our approach is closer to the original approach of Hirzebruch and Zagier \cite{hirzag} for Hilbert modular surfaces. 
\subsection{Comparison with earlier work}
A similar result was proved by Bruinier and Zemel in \cite{bruinierzemel}. In their paper, they analyze the vanishing order of the automorphic forms constructed by Borcherds \cite{borcherds-95-infinite-products}, using an explicit expression of its Petersson metric from \cite{bruinier,borcherds-inventiones-automorphic}. 
After the announcement of our paper, Bruinier pointed out to us \cite{bruinier-letter} that Borcherds' theta lifting method shows that the difference between the holomorphic generating series of \cite{bruinierzemel} and the weakly holomorphic generating series of \ref{t:precise-main} is indeed a weakly holomorphic modular form supported on the type III boundary.

The approach of \cite{bruinierzemel} yields a correction where the coefficients along the type III boundary divisors are a priori not rational, while our approach provides a correction with rational coefficients along those boundary components. By exchanging holomorphic for weakly holomorphic, we gain the rationality of the Fourier coefficients. On the other hand, the completion of \cite{bruinierzemel} is canonical, and ours depends on a choice of preimages for certain unary Theta functions, see \Cref{zagier-series}.

Our desire to understand \cite{bruinierzemel} from an intersection theoretic point of view, in the spirit of Hirzebruch--Zagier's seminal work \cite{hirzag}, was the starting point of our investigations. Our approach is more geometric, and also yields unconditional results in cohomology when $n=2$ and the Witt index of $L$ is $2$. 

Garc\'ia has also recently announced similar results using completely different methods \cite{garcia-kudla-millson}.

There has also been earlier work of \cite{peterson} in the context of moduli space of K3 surfaces of dimension 19 and with compactifications involving only type II boundary components. In the case of modular curves, our work generalizes results of Zagier \cite{zagier-cras} and Funke \cite{funke}.

We are hopeful that this work might extend to modularity results for higher codimension cycles, as well as to other contexts where Borcherds products are not available but Kudla--Millson's results are still valid.



\subsection{Organization of the paper}
In Section 2, we review the Weil representation, introduce mixed mock modular forms, and prove the mixed modularity of theta series with piecewise linear weights on hyperbolic lattices. In Section 3, we review in detail the construction of toroidal compactifications, and prove the main intersection formula of special divisors with toric boundary curves. In Section 4, we prove the splitting theorems for second homology that allow us to reduce the modularity computations to toric boundary curves of the toroidal compactification. We assemble the pieces of the proof of \Cref{t:precise-main} in Section 5.

\subsection{Acknowledgement}
We are grateful to Jan Bruinier and to the referee for pointing out a mistake in Section 2.3. We also thank Jan Bruinier for helpful correspondence, and for explaining the connection to \cite{bruinierzemel}. 
\section{Weil representation and vector-valued modular forms}\label{sec2}
In this section, we recall the Weil representation and define vector-valued mixed mock modular forms following \cite{bruinierfunke}, \cite{folsom}, and \cite{zagier-bourbaki}. Using a theorem of Vign\'eras \cite{vigneras}, we prove that certain theta series of hyperbolic lattices with piecewise linear weights are mixed mock modular forms. Then, we relate them to theta series of rays in the positive cone of the hyperbolic lattice.   
\subsection{Weil representation}
Let $(\Lambda,\cdot)$ be an even lattice of signature $(b^+,b^-)$, with bilinear form 
\begin{align*}
\Lambda\times \Lambda&\rightarrow \Z\\
(x,y)&\mapsto x\cdot y
\end{align*} and quadratic form  $Q(x)=\tfrac{1}{2}{x\cdot x}\in\Z$. Let $\Lambda^\vee$ be the dual lattice of $\Lambda$, defined as 
\[\Lambda^\vee:=\{x\in \Lambda_\Q \,\big{|}\, x\cdot y\in \Z,\,\forall y\in \Lambda\}.\]

The \emph{discriminant lattice} $\Lambda^{\vee}/\Lambda$ is a finite abelian group of cardinality $D_\Lambda=|\det(e_i\cdot e_j)_{i,j}|$ where $\{e_i\}$ form a basis of $\Lambda$. 
It is endowed with a $\Q/\Z$-valued quadratic form: 
\begin{align*}
    Q:\Lambda\times\Lambda\rightarrow \Q/\Z\\
    x\mapsto Q(x).
\end{align*}

Let $\mathrm{Mp}_{2}(\R)$ be the double metaplectic cover of $\mathrm{SL}_{2}(\R)$ whose elements consist of pairs $(M,\phi)$, where $$M=\begin{pmatrix} a&b\\c&d\\ \end{pmatrix}\in \mathrm{SL}_{2}(\R)$$ and $\phi$ is a holomorphic function on the Poincaré upper half-plane $\mathbb{H}$ such  that $\phi(\tau)^2=c\tau +d$, for all $\tau\in \mathbb{H}$. 
The group $\mathrm{Mp}_{2}(\Z)$ is generated by the following elements (see \cite[P.78]{cours}):

$$T=\left(\begin{pmatrix}1&1\\0&1\end{pmatrix},1 \right) \quad \textrm{and}\quad S=\left(\begin{pmatrix}0&-1\\1&0\end{pmatrix},\tau\mapsto\sqrt{\tau} \right).$$

The \emph{Weil representation} $\rho_\Lambda$ associated to $(\Lambda,\cdot)$ is defined as the unique representation \[\rho_\Lambda:\mathrm{Mp}_{2}(\Z)\rightarrow {\rm GL}(\C[\Lambda^\vee/\Lambda])\] 
 such that the elements $S$ and $T$ act in the following way:
\begin{align}\label{weil}
\begin{split}
\rho_{\Lambda}(T)\mathfrak{e}_{\gamma}&=e^{2\pi i Q(\gamma)}\mathfrak{e}_{\gamma},\\
\rho_{\Lambda}(S)\mathfrak{e}_{\gamma}&=\frac{{i}^{\frac{b^{-}-b^{+}}{2}}}{\sqrt{|\Lambda^{\vee}/\Lambda|}}\sum_{\delta\in \Lambda^\vee/\Lambda}e^{-2\pi i (\gamma\cdot\delta)}\mathfrak{e}_{\delta},
\end{split}
\end{align}
where $\gamma\in \Lambda^{\vee}/\Lambda$ and $(\mathfrak{e}_\gamma)_{\gamma\in \Lambda^{\vee}/\Lambda}$ is the canonical basis of $\C[\Lambda^{\vee}/\Lambda]$.

Let $\Lambda_1\subset \Lambda$ be a finite index sublattice. Then we have inclusions 
\[\Lambda_1\subseteq\Lambda\subseteq\Lambda^{\vee}\subseteq\Lambda_1^{\vee}.\]
Let $H=\Lambda/\Lambda_1$. It is a subgroup of $\Lambda_1^{\vee}/\Lambda_1$ and its orthogonal with respect to $\cdot $ is equal to  $\Lambda^{\vee}/\Lambda_1$. Moreover, the subquotient $H^{\bot}/H$ is isomorphic to $\Lambda^{\vee}/\Lambda$, the discriminant lattice of $\Lambda$. Let $p:H^{\bot}\rightarrow \Lambda^{\vee}/\Lambda$ be the projection map. We can then define the following two maps: 
\begin{align*}
\uparrow_{\Lambda_1}^{\Lambda}:\, \C[\Lambda_1^{\vee}/\Lambda_1]&\rightarrow  \C[\Lambda^{\vee}/\Lambda]\\ \numberthis \label{pull}
			\mathfrak{e}_{\gamma}&\mapsto \begin{cases}
                                 \mathfrak e_{p(\gamma)} & \text{if $\gamma \in H^{\bot}$,}\\
                                0 & \text{otherwise,}
                                \end{cases}
\end{align*}
\begin{align*}
 \uparrow_{\Lambda}^{\Lambda_1}:\, \C[\Lambda^{\vee}/\Lambda]&\rightarrow \C[\Lambda_1^{\vee}/\Lambda_1]\\\numberthis\label{push}
			\mathfrak{e}_{\delta}&\mapsto \sum_{\gamma\in H^{\bot},\,p(\gamma)=\delta}\mathfrak{e}_{\gamma}.		 
\end{align*}
These operators intertwine the Weil representations on both sides. 
\medskip

Let $I\subset \Lambda$ be a totally isotropic sublattice. The lattice $I^{\bot}/I$ is a non-degenerate lattice whose discriminant group can be realized as a subquotient of $\Lambda^{\vee}/\Lambda$: Let $\Lambda_I$ be the lattice generated by $\Lambda$ and $I_\Q\cap\Lambda^{\vee}$. Then the discriminant group of $\Lambda_I$ is in fact isomorphic to the discriminant group of $I^{\bot}/I$. Hence we make the following definition.
\begin{definition}\label{d:pull-push-isotropic}
    If $I\subset \Lambda$ is a totally isotropic subspace, then we define
    \[\uparrow_{I^{\bot}/I}^{\Lambda}:=\uparrow_{\Lambda_I}^{\Lambda}\quad\textrm{and}\quad \uparrow^{I^{\bot}/I}_{\Lambda}:=\uparrow^{\Lambda_I}_{\Lambda}~.\]
These operators intertwine the Weil representations on both sides. 
\end{definition}

\subsection{Weak harmonic Maass forms and mixed mock modular forms}\label{s:harmonic-mass-forms}
Let $k\in \frac{1}{2}\Z$. Following \cite[Section 3]{bruinierfunke}, recall the following definition. 
\begin{definition}\label{definition-mass}
A \emph{weak harmonic Maass form} of weight $k$ and representation $\rho_\Lambda$ is a twice-differentiable function $f:\mathbb{H}\rightarrow \C[\Lambda^\vee/\Lambda]$ that satisfies:
\begin{enumerate}
    \item $f(\gamma.\tau)=\phi(\tau)^{2k}\rho_\Lambda(\gamma,\phi)f(\tau)$ for all $(\gamma,\phi)\in\mathrm{Mp}_2(\Z)$;
    \item there exists $C>0$ such that $f(\tau)=O(e^{Cy})$ as $y\rightarrow +\infty$, $\tau=x+iy$ (uniformly in $x$);
    \item $\Delta_kf=0$ where \[\Delta_k:=-y^2\left(\frac{\partial ^2}{\partial x^2}+\frac{\partial ^2}{\partial y^2}\right)+iky\left(\frac{\partial}{\partial x}+i\frac{\partial}{\partial y}\right)\] 
    is the hyperbolic Laplace operator in weight $k$.
\end{enumerate}
\end{definition}

We denote by $H_{k}(\rho_\Lambda)$ the $\C$-vector space of weak harmonic Maass forms of weight $k$ and representation $\rho_\Lambda$. By \cite[Equation (3.2)]{bruinierfunke}, any such form  $f$ has a unique decomposition $f=f^++f^-$ where $f^+$ is holomorphic and $f^-$ is called the \emph{non-holomorphic part} of $f$. 

If $f^-=0$ then $f$ is a weakly holomorphic modular form. Denote by $M^{!}_{k}(\rho_\Lambda)$ the space of weakly holomorphic modular forms of weight $k$ and representation $\rho_\Lambda$.
The \emph{shadow operator} on $H_{k}(\rho_\Lambda)$ is

\begin{align}\label{shadow}
\xi_{k}(f)=2iy^{k}\overline{\partial_{\overline{\tau}}f}\in H_{k}(\overline{\rho}_\Lambda).
\end{align}

By \cite[Prop. 3.2]{bruinierfunke}, $\xi_{k}(f)\in M^{!}_{2-k}(\overline{\rho}_\Lambda)$ and the kernel of $\xi_k$ is equal to $M^{!}_{k}(\rho_\Lambda)$.
The following definition is inspired from \cite[Definition 5.16]{folsom} and is originally due to Zagier \cite{zagier-bourbaki}. We give here a generalization for vector-valued modular forms. 

\begin{definition}
A \emph{vector-valued mock modular form} of weight $k$ is the holomorphic part $f^{+}$ of a weak harmonic Maass form $f$ of weight $k$.
\end{definition}

If $f:\H\rightarrow \C[\Lambda_1^{\vee}/\Lambda_1]$ and $g:\H\rightarrow \C[\Lambda_2^{\vee}/\Lambda_2]$ are two functions which satisfy relation (1) in \Cref{definition-mass}, the function $$f\otimes g:\mathbb{H}\rightarrow \C[\Lambda_1^{\vee}/\Lambda_1]\otimes\C[\Lambda_2^{\vee}/\Lambda_2]$$ can be seen naturally as a function valued on $\C[\Lambda^{\vee}/\Lambda]$ where $\Lambda=\Lambda_1\oplus\Lambda_2$. Then $f\otimes g$ satisfies the same invariance relation with respect to $\rho_{\Lambda^{\vee}/\Lambda}\simeq \rho_{\Lambda_1^{\vee}/\Lambda_1}\otimes \rho_{\Lambda_2^{\vee}/\Lambda_2}$. 
Furthermore, if $\Lambda\subseteq \Lambda'$ is a sublattice of finite index, then we have natural morphisms of vector spaces given by Equation \ref{pull} and Equation \ref{push}: 
\begin{align*} & \uparrow_{\Lambda}^{\Lambda'}:H_{k}(\rho_\Lambda)\rightarrow H_{k}(\rho_{\Lambda'})\textrm{ and} \\
 & \uparrow_{\Lambda'}^{\Lambda}:H_{k}(\rho_{\Lambda'})\rightarrow H_{k}(\rho_\Lambda)\end{align*}
which commute with the shadow operator. In particular, they preserve the holomorphic and weakly holomorphic modular forms. 

The following definition is a slight generalization of Definition 13.1 from \cite{folsom}.
\begin{definition}
A \emph{mixed harmonic Maass form} of weight $(k,\ell)$ and representation $\rho_{\Lambda}$ is a function $h$ of the form 
\[h=\sum_{j=1}^{N}\uparrow_{\Lambda_{j}\oplus \Lambda_j^{\bot}}^{\Lambda}(f_j\otimes g_j),\]
where $\Lambda_j\subseteq \Lambda$ is a non-degenerate sublattice, $f_j$ is a weakly holomorphic vector-valued modular form of weight $k$ with representation $\rho_{\Lambda_j}$ and $g_j$ is vector-valued weak harmonic Maass form of weight $l$ and representation $\rho_{\Lambda_j^{\bot}}$. If $g^+_j$ is the holomorphic part of $g_j$ then by definition, \[h^+=\sum_{j=1}^{N}\uparrow_{\Lambda_{j}\oplus \Lambda_j^{\bot}}^{\Lambda}(f_j\otimes g^+_j),\] is a \emph{mixed mock modular form}. 
\end{definition}

\begin{remarque}
One easily checks that $h^+$ is the holomorphic part of $h$ and hence a different expression of $h$ yields the same function $h^+$.    
\end{remarque}


If we fix a weight $k\in \frac{1}{2}\Z$ and a representation $\rho_\Lambda$, then we can also define mixed mock modular forms of total weight $k$ as follows. 
\begin{definition}\label{d:mixed-mock-modularity}
A \emph{mixed harmonic Maass form} of weight $k$ and representation $\rho_\Lambda$ is a function which can be expressed as a finite sum 
\[\sum_{k_1+k_2=k}h_{k_1,k_2}\] where each $h_{k_1,k_2}$ is a mixed harmonic Maass form of weight $(k_1,k_2)$ and representation $\rho_\Lambda$. By definition, its holomorphic part $\sum_{k_1+k_2=k}h^+_{k_1,k_2}$ is a mixed mock modular form of weight $k$ and representation $\rho_\Lambda$.
\end{definition}



\subsection{Zagier's Eisenstein series and generalizations}\label{s:zagier}

Let $N\geq 1$ and assume in this section that $\Lambda=\Z x$ is endowed with the bilinear form $x\cdot x=2N$. The discriminant lattice is then isomorphic to $\Z/2N\Z$. 

The vector-valued theta series associated to $\Lambda$ is expressed as
\begin{align}\label{theta-half}
    \Theta_{\frac{1}{2},N}=\sum_{n\in \Z}q^{\tfrac{n^2}{4N}}\mathfrak{e}_{[n]}\in M_{\frac{1}{2}}(\rho_N)~,
\end{align} where $[n]\in \Z/2N\Z$ is the congruence class.
It is a holomorphic modular form of weight $\frac{1}{2}$ and representation $\rho_{N}$.

Recall the shadow operator of Equation \ref{shadow} is
\[\xi_{\frac{3}{2}}: H_{\frac{3}{2}}(\overline{\rho}_N)\rightarrow M^{!}_{\frac{1}{2}}(\rho_N)~.\]
In their seminal paper \cite{hirzag}, see also \cite{zagier-cras}, \cite[Theorem 6.3]{folsom}, Hirzebruch and Zagier construct for $N=1$ a particular weak harmonic Maass form $F_1\in H_{\frac{3}{2}}(\overline{\rho}_1)$, {\it Zagier's Eisenstein series}, satisfying \[\xi_{\frac{3}{2}}\left(F_1\right)=-\frac{1}{8\pi}\Theta_{\frac{1}{2},1}.\] 
The holomorphic part of this form has the important property that all its Fourier coefficients are rational numbers of arithmetic significance. In what follows, we recall the construction for general $N\geq 1$ following the work of Bruinier and Schwagenscheidt, see \cite[Theorem 4.3]{bruinier-sch}.

Let $r\in \Z/2N\Z$ and let $D \in \Z_{<0}$ such that $D\equiv r^2\,(4N)$. Let $Q_{N,D,r}$ be the set of integral binary quadratic forms $Q:x\mapsto ax^2+bxy+cy^2$ of discriminant $D=b^2-4ac$ with $N|a$ and $b\equiv r\,(2N)$. The group $\Gamma_{0}(N)$ acts on this set with finitely many orbits  and the order of a stabilizer $w_Q=\frac{1}{2}|\mathrm{Stab_Q}(\Gamma_{0}(N))|$ is finite. For each such binary form, one can define a CM point $z_Q=\frac{-b+\sqrt{D}}{2a}\in \H$ and which satisfies the equation $Q(z_Q,1)=0$.\\ 

Let $M^{!,\infty}_{0}(N)$ be the space of weakly holomorphic modular forms of weight $0$, of level $N$ and which vanish at all cusp except $\infty$. Let $F\in M^{!,\infty}_{0}(N)$ be a form given by \cite[Lemma 4.2]{bruinier-sch}. We can write such form as $F=\sum_{n\gg-\infty}a(n)q^n$, the $a(n)$ are rational coefficients and $a(0)=\frac{1}{2}$. Define 
\[H(D,r,F):=\sum_{Q_{N,D,r}/\Gamma_{0}(N)}\frac{F(z_Q)}{w_Q}~.\]
For $t\geq 0$, let $\beta(t)=\frac{1}{2\pi}\int_{1}^{\infty}u^{-\frac{3}{2}}e^{-\pi t u}du$ and let $(\mathfrak{e}_r)_{r\in \Z/2N\Z}$ be the canonical basis of $\C[\Lambda^\vee/\Lambda]$. We define

\begin{multline}
    F_N(\tau)=4\sum_{n\geq 0}a(-n)\sigma_{1}(n)\mathfrak e_0+\sum_{r\in \Z/2N\Z}\,\sum_{\substack{D\equiv r^2\,(4N) \\ D<0}}^{\infty}H(D,r,F)q^{-\tfrac{D}{4N}}\mathfrak e_r\\
    -\sum_{m>0}\sum_{n=1}^{\infty}ma(-mn)q^{-\tfrac{m^2}{4N}}(\mathfrak{e}_{[m]}+\mathfrak{e}_{[-m]})\\
    +\frac{\sqrt{N}}{\sqrt{y}}\sum_{n\in \Z}\beta\left(\frac{yn^2}{N}\right)q^{-\tfrac{n^2}{4N}}\mathfrak{e}_{[n]}.
\end{multline}

Notice in particular that the principal part of $F_N$ is non-trivial in general, unless $N=1$ (this recovers the series constructed by Zagier) or if $N$ is prime, see \cite[Remark 4.4]{bruinier-sch}. By \cite[Theorem 4.3 (1)]{bruinier-sch}, we have the following lemma. 

\begin{proposition}\label{zagier-series}
The series $F_N$ is an element of $H_{\frac{3}{2}}(\overline{\rho}_N)$. Moreover, 
\[\xi_{\frac{3}{2}}(F_{N})=-\frac{\sqrt{N}}{8\pi}\Theta_{\frac{1}{2},N}.\]
\end{proposition}

\subsection{Vignéras' modularity criterion}

We recall in this section Vignéras' modularity criterion \cite[Th\'eor\`eme 1]{vigneras}, see also \cite[Theorem 2.3]{funkekudla}, which will be our main tool for constructing mixed mock modular forms. 

Let $(\Lambda,Q)$ be a quadratic space of signature $(p,q)$. If we choose a basis of $\Lambda_\R$ in which $Q$ has the shape $Q(x)=\frac{1}{2}\left(\sum_{i=1}^{p}x_i^2-\sum_{i=p+1}^{p+q}x_i^{2}\right)$, then the Euler differential operator $E$ and the Laplace differential operator $\Delta$ associated to $-Q$ are:

\[E=\sum_{i=1}^{p+q}x_i\frac{\partial}{\partial x_i},\qquad \Delta=-\sum_{i=1}^{p}\frac{\partial^2}{\partial x_i^2}+\sum_{i=p+1}^{p+q}\frac{\partial^2}{\partial x_i^2}.\]

We have the following theorem, see \cite[Th\'eor\`eme 1]{vigneras} applied to the quadratic space $(\Lambda,-Q)$.
\begin{theorem}\label{vigneras}
Let $p$ be a twice differentiable function on $\Lambda_\R$ that satisfies the following assumptions: 
\begin{enumerate}
    \item The differential equation \[(E-\tfrac{1}{4\pi}\Delta)p=kp\] holds for an integer $k$.
    \item Let $f:x\mapsto p(x)e^{2\pi Q(x)}$. Then $D(f)$ and $x\mapsto R(x)f(x)$ are in $L^2(\R^{p+q}) \cap L^{1}(\R^{p+q})$ for every derivation $D$ of order at most $2$ and every polynomial $R$ of degree at most $2$.
\end{enumerate} 
Then the series 
\[y^{\frac{k}{2}}\sum_{\lambda\in \Lambda^{\vee}} p(\lambda\sqrt{y})q^{-Q(\lambda)}\mathfrak{e}_{[\lambda]}, \quad q=e^{2\pi i\tau},\]
converges absolutely and defines a non-holomorphic modular form of weight $k+\frac{p+q}{2}$ and representation $\overline{\rho}_\Lambda$.
\end{theorem}

Here $[\lambda]\in K^\vee/K$ denotes the class of $\lambda$ in the discriminant group.

\begin{remarque}
In \cite[Th\'eor\`eme 1]{vigneras}, the above theorem is stated for scalar modular forms by introducing a level and a character, but the vector-valued version follows easily from the proof, especially formulas (1) and (3) in {\it loc. cit.}.
\end{remarque}

\subsection{Hyperbolic lattices}\label{hyperboliclattices}
We specialize our discussion from before to hyperbolic lattices. Let $(K,\cdot)$ be an even lattice of signature  $(1,n-1)$. Let $C_K$ be a connected component of $\{x\in K_{\R}\,|\, x\cdot x>0\}$. Then for any $v_1$, $v_2$ in $C_K$, we have  $v_1\cdot v_2>0$.

Let $\mathcal{E}(x)=2\int_0^x e^{-\pi t^2}dt$ be the Gauss error
function and define $\psi(x)=x\mathcal{E}(x)+\frac{1}{2\pi}\mathcal{E}'(x)$. Notice that $\psi'=\mathcal E$ and one can easily check that $\psi$ satisfies the following differential equation: 
\begin{align}\label{eigenvalue}
    x\psi'(x)+\frac{1}{2\pi}\psi''(x)=\psi.
\end{align}

Moreover, we have 
\[\mathcal{E}(x)=\sgn(x)(1-2\int_{|x|}^\infty e^{-\pi t^2}dt),\]
and we can thus write:
\begin{align}\label{beta_decomposition}
    \psi(x)=x\mathcal{E}(x)+\frac{1}{2\pi}\mathcal{E}'(x)=|x|+\beta(x^2),
\end{align}
where for $t\geq 0$, \[\beta(t)=\frac{1}{2\pi}\int_1^{\infty}u^{-\frac{3}{2}}e^{-\pi tu}du.\]

The importance of the functions $\mathcal E$ and $\psi$ is that they are smooth approximations of the sign and the absolute value functions respectively, and they will help us build smooth approximations to the piecewise linear functions appearing in the intersection formulas of special cycles. 

For a collection of vectors $\{c_i\}_{i\in S}$ in the closure of the positive cone $C_K$ that satisfy $\sum_{i\in S} a_ic_i=0$, define the function $$p^+(\lambda):=\sum_{i\in S}a_i|\lambda\cdot c_i|$$ and consider the formal power series: 
\begin{align}\label{eq:main-power-series}
\Theta^+=\sum_{\lambda\in K^{\vee}}p^+(\lambda)q^{-Q(\lambda)}\mathfrak{e}_{[\lambda]}\in\C[K^{\vee}/K][[q^{\frac{1}{D_K}}]].
\end{align}
Our goal in the next two sections is to analyze the convergence and the modularity properties of this series. 

\subsection{Case of non-isotropic vectors}

Let $c\in C_K$ and define  
\[\phi_c(\lambda):=\sqrt{\frac{c\cdot c}{2}}\psi\left(\frac{\sqrt{2}\, \lambda\cdot c}{\sqrt{c\cdot c}}\right)\textrm{ for }\lambda\in K_\R. \]
Then using \Cref{eigenvalue}, one can check that: 
\begin{align}\label{Eq:eigen_value}
    \left(E-\tfrac{1}{4\pi}\Delta \right)\phi_c=\phi_c.
\end{align}

Now consider the special case where $(c_i)_{i\in S}$ is a finite set of primitive integral elements of the positive cone $C_K\cap K$, such that $\sum_{i\in S}a_ic_i=0$, with $a_i\in\Z$ non-zero integers. Let $c_i\cdot c_i=2N_i>0$. Define 
\begin{align*}
    p(\lambda):&=\sum_{i\in S}a_i\phi_{c_i}(\lambda)\\\numberthis \label{completion}
            &=\sum_{i\in S}a_i\sqrt{N_i}\psi\left(\frac{\lambda\cdot c_i}{\sqrt{N_i}}\right).
\end{align*}
\begin{proposition}\label{p:convergence}
Let $c_i$ be a collection of (non-isotropic) primitive integral vectors of $C_K$ which satisfy $\sum_{i}a_ic_i=0$ for some integers $a_i$. Then the function $p$ satisfies the conditions of \Cref{vigneras}.

\end{proposition}
The proof of this proposition is inspired from the proof of Proposition 2.4 in \cite{zwegers}. We first prove the following lemma. 

\begin{lemma}\label{l:euclidean}
Let $(c_i)_{i\in S}$ be a finite set of distinct integral elements of $\overline{C}_K$ which satisfies a relation of the form $\sum_ia_ic_i=0,$ where $a_i$ are non-zero numbers and let $p^+(\lambda)=\sum_{i}a_i|\lambda\cdot c_i|$. Let $U$ denote the open set of elements $\lambda\in K_\R$ where the function $\lambda\mapsto p^+(\lambda)$ does not vanish. Then there exists an euclidean norm $||\cdot ||$ on $K_\R$ such that:
\begin{enumerate}
    \item If all the vectors $c_i$ are in $C_K$, then \[\forall \lambda\in U,\quad  -Q(\lambda)\geq ||\lambda||^2.\]
    \item In general, 
    \[\forall \lambda\in U\cap K,\quad  -Q(\lambda)\geq ||\lambda||.\]
\end{enumerate}

\end{lemma}
\begin{proof}
Notice that $S$ has cardinality at least $3$ and that whenever $\lambda\cdot c_i$ have the same sign for all $i\in S$, then the function $\lambda\mapsto p^+(\lambda)$ vanishes. Hence $U$ is contained in the set where the linear forms $\lambda\mapsto \lambda\cdot c_i$ don't have all the same sign. In particular, $U$ is a finite union over $(i,j)$ of the sets $U_{(i,j)}$ where $\lambda\cdot c_i\geq 0$ and $\lambda\cdot c_j\leq 0$. So it suffices to prove the lemma for each $U_{(i,j)}$. Fix such pair and call it $(1,2)$. 

{\bf First case:} Assume that $c_1$ and $c_2$ are not isotropic. Write \begin{align*} \lambda&=\lambda_{c_1^{\bot}}+\frac{\lambda\cdot c_1}{c_1\cdot c_1}c_1. \textrm{ Then} \\
0\geq \lambda\cdot c_2&=\lambda_{c_1^{\bot}}\cdot c_2+\frac{\lambda\cdot c_1}{c_1\cdot c_1}c_1\cdot c_2\implies \\
& \hspace{-15pt}\frac{(\lambda\cdot c_1)(c_1\cdot c_2)}{c_1\cdot c_1}
\leq -\lambda_{c_1^{\bot}}\cdot c_2.\end{align*}

Let $\widetilde{c}_2$ be the orthogonal projection of $c_2$ to $c_1^{\bot}$. By Cauchy--Schwarz inequality, we have \[(\lambda_{c_1^{\bot}}\cdot c_2)^2=(\lambda_{c_1^{\bot}}\cdot \widetilde{c}_2)^2\leq (\lambda_{c_1^{\bot}}\cdot \lambda_{c_1^{\bot}})(\widetilde{c}_2\cdot \widetilde{c}_2).\]

Hence we get 
\begin{align*}
    \frac{(\lambda\cdot c_1)^2}{|c_1\cdot c_1|}&\leq \frac{|c_1\cdot c_1|}{(c_1\cdot c_2)^2}(\lambda_{c_1^{\bot}}\cdot \lambda_{c_1^{\bot}})\left(c_2\cdot c_2-\frac{(c_1\cdot c_2)^2}{c_1\cdot c_1}\right)\\
    &= \left(1-\frac{(c_2\cdot c_2)(c_1\cdot c_1)}{(c_1\cdot c_2)^2}\right)|\lambda_{c_1^{\bot}}\cdot\lambda_{c_1^{\bot}}|.
\end{align*}
Let $\mu>0$ such that $\frac{1-\mu}{1+\mu}=1-\frac{(c_2\cdot c_2)(c_1\cdot c_1)}{(c_1\cdot c_2)^2}$. Then one can check that 

\begin{align}\label{norm-square}
-Q(\lambda)\geq \mu Q_{c_1}(\lambda),
\end{align} where $Q_{c_1}$ is a positive definite quadratic form defined by: \[Q_{c_1}(\lambda):=-Q(\lambda_{c_1^{\bot}})+\frac{(\lambda\cdot c_1)^2}{2(c_1\cdot c_1)}.\]

{\bf Second case:} Assume now that $c_1$ is isotropic and $c_2$ and is still anisotropic. The lattice $(c_1,c_2)$ has signature $(1,1)$ so its orthogonal $M\subset K$ is negative definite.
We can write \[\lambda=n_1c_1+n_2c_2+v,\] where $v\in M^\vee$. Then by assumption $\lambda\cdot c_1=n_2(c_1\cdot c_2)\geq 0$ which implies that $n_2\geq 0$ and $\lambda\cdot c_2=n_1(c_1\cdot c_2)+n_2(c_2\cdot c_2)\leq 0$. So 
\[n_2(c_2\cdot c_2)\leq -n_1(c_1\cdot c_2),\]
and in particular, $n_1\leq 0$. 

On the other hand, 
\begin{align*}-Q(\lambda)&=-n_1n_2(c_1\cdot c_2)-n_2^{2}\frac{c_2\cdot c_2}{2}-Q(v)\\
&= -\frac{n_1n_2(c_1\cdot c_2)}{2}-n_2\frac{n_1(c_1\cdot c_2)+n_2(c_2\cdot c_2)}{2}-Q(v)\\
&\geq  \frac{n_2^2|c_2\cdot c_2|}{2}-Q(v).
\end{align*}
If $n_2\neq 0$, then since $\lambda$ and $c_1$ are in $K$, we have that $|n_2(c_1\cdot c_2)|\geq 1$\footnote{If we assume that $c_i$ are rational, then we have to modify $1$ by a constant here and the rest of proof holds.}, hence we get also $-Q(\lambda)\geq \frac{|n_1|}{2}$ and therefore 
\[-3Q(\lambda)\geq |n_1|+n_2^2|c_2\cdot c_2|-Q(v)\gg ||\lambda||\]
for the euclidean norm $||\lambda||^2=n_1^2+n_2^2-Q(v)$. 

If $n_2=0$, then there exists $c_3$ such that $\lambda\cdot c_3>0$, otherwise all $\lambda\cdot c_i$ would have the same sign for $i\neq 1$ and hence \[p(\lambda)=2a_1|\lambda\cdot c_1|=2a_1|n_2(c_1\cdot c_2)|,\] contradicting the fact that $p$ does not vanish on $\lambda$. From the inequality $\lambda\cdot c_3>0$ we get then: 
\[n_1(c_1\cdot c_3)+n_2(c_2\cdot c_3)+\widetilde{c_3}\cdot v>0,\]
where $\widetilde{c_3}$ is the projection of $c_3$ to  $M$. From Cauchy-Schwartz inequality, we have: 
\[|n_1(c_1\cdot c_3)|\leq |n_2(c_2\cdot c_3)|+\sqrt{(v\cdot v)(\widetilde{c_3}\cdot \widetilde{c_3})},\]
from which it follows that there exists a constant $\mu_1>0$ such that 
\[-Q(\lambda)\geq \mu_1\left(n_1^2+n_2^2-Q(v)\right),\]
and we can choose again the euclidean norm $||\lambda||^2=n_1^2+n_2^2-Q(v)$.

{\bf Last case:} Assume now that both $c_1$ and $c_2$ are isotropic. Then we can introduce $c_2'=c_1+c_2$ which is not isotropic and notice that $U_{1,2}$ is contained in $U_{1,2'}\cup U_{2',2}$ and hence we use the previous analysis on each of them. This finishes the proof of the lemma.
\end{proof}

We are now ready to prove Proposition \ref{p:convergence}. 
\begin{proof}[Proof \Cref{p:convergence}]
Condition (1) follows from Equation \ref{Eq:eigen_value}. We only need to prove condition (2). By Equation \ref{beta_decomposition}, we have the decomposition for $x\in\R$ 
\[\psi(x)=|x|+\beta(x^2).\]
This corresponds to a decomposition
\[p(\lambda)=p^{+}(\lambda)+p^{-}(\lambda),\] 
where $p^{+}(\lambda)=\sum_{i}a_i|\lambda\cdot c_i|$. From the estimate $\beta(t)\ll e^{-\pi t }$, for $t\geq 0$, we get that 
\begin{align*}
    |p^{-}(\lambda)e^{\pi(\lambda\cdot\lambda)}|&\ll\sum_{i\in S} e^{-\pi\big[-\lambda\cdot\lambda+2\tfrac{(\lambda\cdot c_i)^2}{c_i\cdot c_i}\big]}\\
    &\ll \sum_{i\in S} e^{-2\pi Q_{c_i}(\lambda)},
\end{align*}
where as before, $Q_c(\lambda)=-Q(\lambda_{c^{\bot}}) +\tfrac{(\lambda\cdot c)^2}{2(c\cdot c)}$ is the positive quadratic form associated to a positive integral vector $c$.

The function $p^+$ vanishes if all the linear forms $\lambda\mapsto \lambda\cdot c_i$ have the same sign when $i\in S$. Hence, in the region where it is non-vanishing, by (1) in \Cref{l:euclidean}, there exists an Euclidean norm on $K_\R$ such that: 
\[|p^{+}(\lambda)e^{\pi(\lambda\cdot\lambda)}|\leq |p^{+}(\lambda)e^{-\pi||\lambda||^2}|.\]
Combining the two previous estimates, this proves that for any polynomial $R$, the function $\lambda\mapsto R(\lambda)p(\lambda)e^{\pi(\lambda\cdot\lambda)}$ is in $L^2(\R^{n}) \cap L^{1}(\R^{n})$.

For the first order derivatives, notice that 
\[\psi'(x)=\mathcal{E}(x)={\rm sgn}(x)-{\rm sgn}(x)\cdot 2\int_{|x|}^{\infty}e^{-\pi t^2}dt.\] 

For any non-zero vector $v\in K_\R$, the  partial derivative $\partial_{v}p$ can then be decomposed accordingly as \[(\partial_{v}p)^+(\lambda)+(\partial_{v}p)^{-}(\lambda).\]

One can easily check that the function $(\partial_{v}p)^+$ vanishes on the region where all the $\lambda\mapsto \lambda\cdot c_i$ have the same sign. Furthermore, we have the estimate 
\[\int_{|x|}^{\infty}e^{-\pi t^2}dt\ll e^{-\pi x^2},\]
which implies in a similar way that

\begin{align*}
    |(\partial_{v}p)^{-}(\lambda)e^{\pi(\lambda\cdot\lambda)}|\ll
     \sum_{i\in S} e^{-2\pi Q_{c_i}(\lambda)}.
\end{align*}
For the second order derivative, since $\psi''(x)=e^{-\pi x^2}$, we get immediately for any two non-zero vectors $v,v'\in K_{\R}$ the estimate
\begin{align*}
    |\partial_v\partial_{v'}p(\lambda)e^{-\pi(\lambda\cdot\lambda)}|\ll\sum_{i\in S} e^{-\pi\big[\lambda\cdot\lambda-2\tfrac{(\lambda\cdot c_i)^2}{c_i\cdot c_i}\big]}\\
    \leq \sum_{i\in S} e^{-2\pi Q_{c_i}(\lambda)}.
\end{align*}
Hence the result. 
\end{proof}

As it appeared in the previous proof, the function $p$ is a smooth approximation of the function $p^{+}(\lambda)=   \sum_{i\in S}a_i|\lambda\cdot c_i|$.
By \Cref{vigneras} the series
\begin{align*}
    \Theta(\tau)&=\frac{1}{\sqrt{y}}\sum_{\lambda\in K^{\vee}}p(\sqrt{y}\lambda)q^{-Q(\lambda)}\mathfrak{e}_{[\lambda]}\\
    &=\frac{1}{\sqrt{y}}\sum_{\gamma\in K^\vee/K}\sum_{\lambda\in \gamma+K}p(\sqrt{y}\lambda)q^{-Q(\lambda)}\mathfrak{e}_\gamma
\end{align*}
transforms like a non-holomorphic modular form of weight $1+\frac{n}{2}$ and representation $\overline{\rho}_K$, where $(\mathfrak{e}_\gamma)_{\gamma\in K^\vee/K}$ is the usual basis of the $\C$-vector space  $\C[K^\vee/K]$.
\medskip

The holomorphic part of this series is equal to the series introduced in \Cref{eq:main-power-series}: 
\[\Theta^{+}(\tau)=\sum_{\lambda \in K^\vee} p^{+}(\lambda)q^{-Q(\lambda)}\mathfrak{e}_{[\lambda]}.\]
The non-holomorphic part is equal to

\begin{align*}
\sum_{i\in S}a_i \Theta^{c_i}_K(\tau),
\end{align*}
where $c_i\cdot c_i=2N_i$ and the series $\Theta^{c_{i}}_K$ is an absolutely convergent series given as follows: 

\begin{align*}
    \Theta^{c_i}_K(\tau)&=\frac{\sqrt{N_i}}{\sqrt{y}}\sum_{\lambda\in K^\vee}\beta\left(\frac{y(\lambda\cdot c_i)^2}{N_i}\right)q^{-Q(\lambda)}\mathfrak{e}_{[\lambda]}\\
    &=\frac{\sqrt{N_i}}{\sqrt{y}}\uparrow_{c_i^\perp\oplus \Z c_i}^{K}\Bigg(\sum_{\substack{\lambda=\delta+\tfrac{n}{2N_i}c_i  \\ \delta\in (c_i^\perp)^\vee,\, n\in \Z}}\beta\left(\frac{yn^2}{N_i}\right)q^{-\tfrac{n^2}{4N_i^2}Q(c_i)-Q(\delta)}\mathfrak{e}_{[\delta]}\otimes \mathfrak{e}_{[n]}\Bigg)\\
    &=\frac{\sqrt{N_i}}{\sqrt{y}}\uparrow_{c_i^\perp\oplus \Z c_i}^{K}\Bigg[\Bigg(\sum_{\delta\in (c_i^\perp)^\vee} q^{-Q(\delta)}\mathfrak{e}_{[\delta]} \Bigg)\otimes \Bigg(\sum_{n\in \Z}  
    \beta\left(\frac{yn^2}{N_i}\right)q^{-\tfrac{n^2}{4N_i}}\mathfrak{e}_{[n]}\Bigg)\Bigg] \\
    &=\uparrow_{c_i^\perp\oplus \Z c_i}^{K}\left(\Theta_{c_i^{\bot}}\otimes 
    F_{N_i}^-\right)
\end{align*}
Here $\Theta_{c_i^{\bot}}$ is the theta series of the negative definite lattice $c_i^\perp$ which is a holomorphic modular form of weight $\frac{n-1}{2}$ and representation $\overline{\rho}_{c_i^\perp}$ and $F_{N_i}^-$ is the non-holomorphic part of the generalized Zagier Eisenstein series introduced in \Cref{s:zagier}. 
In particular, the non-holomorphic part of 
\[\Theta-\sum_{i\in S}a_i\uparrow_{c_i^\perp\oplus \Z c_i}^{K}(\Theta_{c_i^{\bot}}\otimes F_{N_i}).\]
vanishes.
Therefore, it is a (weakly) holomorphic modular form of weight $1+\frac{n}{2}$ and representation $\overline{\rho}_K$.
We have thus proven the following theorem.

\begin{theorem}\label{completion-type-III}
Let $(c_i)_{i\in S}$ be a finite set of integral vectors of $C_K$ with a relation $\sum_{i\in S}a_ic_i=0$, $a_i\in\Z$. Then the function $\Theta^{+}$ is a mixed harmonic Maass form of weight $1+\frac{n}{2}$ and representation $\overline{\rho}_K$. More precisely, the function 
\[\Theta^+-\sum_{i\in S}a_i\uparrow_{c_i^\perp\oplus \Z c_i}^{K}(\Theta_{c_i^{\bot}}\otimes F^+_{N_i}).\]
is a (weakly) holomorphic modular form of weight $1+\frac{n}{2}$ and representation $\overline{\rho}_K$. 
\end{theorem}

\subsection{Case of isotropic vectors}
We now analyze the general case where the vectors $c_i$ are allowed to be isotropic. For each isotropic vector $c_i$, let $M_i:=c_i^\bot/c_i$ be the associated negative definite lattice.

We can write \[\sum_{i\in S_1}a_ic_i=-\sum_{i\in S_2}a_i c_i,\]
where $S_1$ is the set of indices where $a_i>0$ and $S_2$ is the set of indices where $a_i<0$. Define $c:=\sum_{i\in S_1} a_ic_i$, so that $c$ is an integral vector in the positive cone $C_K$. 

For $\epsilon>0$, define $$c_i(\epsilon):=\twopartdef{c_i+\epsilon \frac{c}{\sum_{\ell\in S_1} a_\ell}}{i\in S_1}{c_i-\epsilon \frac{c}{\sum_{\ell\in S_2} a_\ell}}{i\in S_2.}$$ Notice that the relation $\sum_ia_ic_i(\epsilon)=0$ still holds and for $\epsilon$ small enough, the vectors $c_i(\epsilon)$ all lie in the positive cone $C_K$.



For $\lambda\in K_\R$, let $p^+(\lambda)=\sum_{i}a_i|\lambda\cdot c_i|$ and consider again the series: 
\[\Theta^+=\sum_{\lambda\in K^{\vee}}p^{+}(\lambda)q^{-Q(\lambda)}\mathfrak{e}_{[\lambda]}\in\C[K^{\vee}/K][[q^{\frac{1}{D_K}}]].\]

Let $E_2(\tau)=1-24\sum_{n\geq 1}\sigma_1(n)q^n$. It is a quasi-modular form that admits the completion $E_2^*(\tau)=E_2(\tau)-\frac{3}{\pi y}$ into a non-holomorphic modular form of weight $2$ with respect to $\mathrm{SL}_{2}(\Z)$. Our goal in this section is to prove the following theorem. 
\begin{theorem}\label{typeII+typeIII}
Let $(c_i)_{i\in S}$ be a finite set of integral vectors of $\overline{C}_K$ with a relation $\sum_{i}a_ic_i=0$, $a_i\in\Z$. Then the series $\Theta^+$ converges absolutely for $\tau\in \mathbb H$ where $q=e^{2 \pi i\tau}$ to a holomorphic function which is a mixed mock modular form. More precisely, the function 
    \[ \Theta^+-\sum_{Q(c_i)>0} a_i \uparrow_{c_i^\perp \oplus \Z c_i}^{K}(\Theta_{c_i^{\bot}}\otimes F^+_{N_i})-\frac{1}{6}\sum_{Q(c_i)=0}a_iE_2 \uparrow_{D}^{K}\big( \Theta_{c_i^{\bot}/c_i}\big)\]
    is a (weakly) holomorphic modular form of weight $1+\frac{n}{2}$ and representation $\overline{\rho}_K$. 
\end{theorem}

\begin{proof}
We first prove the convergence of $\Theta^+$: by \Cref{l:euclidean}, there exists an euclidean norm $||\cdot||$ such that in the region $U\cap K^{\vee}$ where $\lambda\mapsto p(\lambda)$ does not vanish, we have $-Q(\lambda)\geq ||\lambda||$.
So for every $m\in \frac{1}{D_K}\Z$, the number of vectors $\lambda\in K^{\vee}$ such that $|Q(\lambda)|\leq m$ is finite and upper bounded by a polynomial in $m$. This implies the uniform absolute convergence of $\Theta^+$ on compact subsets of $\mathbb H$ to a holomorphic function on $\H$. 

To obtain the mixed mock modularity, we approach $\Theta^+$ with a sequence of mixed mock modular theta series. Let $\epsilon>0$ be small enough so that all the vectors $c_i(\epsilon)$ are in the positive cone $C_K$. 

Consider the holomorphic function on $\mathbb{H}$:

\[\Theta_\epsilon(\tau)=\sum_{\lambda\in K^{\vee}}p_\epsilon^{+}(\lambda)q^{-Q(\lambda)}\mathfrak{e}_{[\lambda]}\in\C[K^{\vee}/K][[q^{\frac{1}{D_K}}]]\]
where $p_\epsilon^{+}(\lambda):=\sum_{i}a_i|\lambda\cdot c_i(\epsilon)|.$
The above series is absolutely convergent and uniformly convergent on compact subsets of $\mathbb H$. Notice also that, as $\epsilon\rightarrow 0$, $p_\epsilon(\lambda)\rightarrow p(\lambda)$ for every $\lambda\in K^{\vee}$. 

\begin{lemma}\label{l:holo-part}
    As $\epsilon \rightarrow 0$, $\Theta^{+}_{\epsilon}(\tau)\rightarrow \Theta^+(\tau)$ uniformly in $\tau$ in compact regions of $\H$.
\end{lemma}
\begin{proof}

Note that in the region where $(\lambda\mapsto\lambda\cdot c_i)_{i\in S}$
all have the same sign, the functions $p^+$ and $p_\epsilon$ both vanish. Indeed, if all the linear forms $\lambda\mapsto \lambda\cdot c_i$ have the same sign, then $\lambda\mapsto \lambda\cdot c$ also has the same sign as them, and hence for any $\epsilon>0$, the linear forms $\lambda\mapsto\lambda\cdot c_i(\epsilon)$ have the same sign. Hence the function $p^+_\epsilon$ vanishes.  Let $U$ be the complement of the above region. Then, using \Cref{l:euclidean}, we have for any $\tau \in \mathbb{H}$ with $\mathrm{Im}(y)>B$: 
\begin{align*}
\left|\left|\Theta^{+}_\epsilon(\tau)-\Theta^+(\tau)\right|\right|&\leq \sum_{\lambda\in K^{\vee}\cap\,U}\sum_{i\in S} |a_i|\cdot\bigg||\lambda\cdot c_i|-|\lambda\cdot c_i(\epsilon)|\bigg|e^{2\pi y Q(\lambda)}\\
&\leq \sum_{\lambda\in K^{\vee}\cap\,U}\sum_{i\in S} |a_i|\cdot|\lambda\cdot c_i-\lambda\cdot c_i(\epsilon)|e^{2\pi y Q(\lambda)}\\
&\ll \epsilon \!\!\sum_{\lambda\in K^{\vee}\cap \,U}\sum_{i\in S} |a_i|\cdot|\lambda\cdot c|e^{-2\pi B ||\lambda||}~.
\end{align*}
    Hence the result. 
\end{proof}




Let $p_\epsilon$ be the completion of $p^{+}_\epsilon$ as in \Cref{completion}. Then by \Cref{p:convergence}\footnote{We proved this proposition when the vectors $c_i$ are integral vectors $c_i$ but the proof holds also when they are rational, as we can assume $\epsilon$ to be a rational number.} and \Cref{vigneras}, the function: 

\[\widehat{\Theta}_\epsilon(\tau)=\frac{1}{\sqrt{y}}\sum_{\lambda\in K^{\vee}}p_{\epsilon}(\sqrt{\lambda})q^{-Q(\lambda)}\mathfrak{e}_{[\lambda]}\]
transforms like a non-holomorphic modular form of weight $1+\frac{n}{2}$ and representation $\overline{\rho}_K$. As $\epsilon \rightarrow 0$, its holomorphic part converges to $\Theta^+$ by \Cref{l:holo-part} and we will analyze the convergence of its non-holomorphic part. For $\epsilon>0$, the non-holomorphic part is given as the sum over $i$ of terms of the form: 
\begin{align*}
    \Theta^{c_i(\epsilon)}_{K}(\tau)=\sqrt{\frac{N_i(\epsilon)}{y}}\sum_{\lambda\in K^{\vee}}\beta\left(\frac{y(\lambda\cdot c_i(\epsilon))^2}{N_i(\epsilon)}\right)q^{-Q(\lambda)}\mathfrak{e}_{[\lambda]}.
\end{align*}
If $c_i$ is not isotropic, then, as $\epsilon\rightarrow 0$, $N_i(\epsilon)\rightarrow N_i \neq 0$ and the quadratic form $\lambda\mapsto \frac{(\lambda\cdot c_i(\epsilon))^2}{N_i(\epsilon)}-2Q(\lambda)$ limits to the positive definite quadratic form $2Q_{c_1}$, hence there exists $A>0$ depending only on $c$ and $c_1$ such that for $\epsilon\leq A$, we have for every $\lambda \in K_\R$ the inequality
\[ \frac{(\lambda\cdot c_i(\epsilon))^2}{N_i(\epsilon)}-2Q(\lambda)\geq Q_{c_1}(\lambda)\geq 0.\]

Hence, using the estimate $\beta(t)\ll e^{-\pi t}$, we get 
\begin{align*}
 \left|\beta\left(\frac{y(\lambda\cdot c_i(\epsilon))^2}{N_i(\epsilon)}\right)q^{-Q(\lambda)}\right|&
\ll \frac{1}{2\pi}e^{-\pi y\left(\tfrac{(\lambda\cdot c_i(\epsilon))^2}{N_i(\epsilon)}-2Q(\lambda)\right)}\\
&\ll e^{-\pi y Q_{c_1}(\lambda)}.
\end{align*}

This implies the uniform convergence and hence we can interchange the limit as $\epsilon\rightarrow 0$ with the sum over $\lambda\in K^{\vee}$ yielding the convergence of the function $\Theta^{c_i(\epsilon)}_K$ to $\Theta_K^{c_i}$ introduced in the previous section. 
\bigskip

We now identify the limit of $\Theta^{c_i(\epsilon)}_K$ when $c_i$ is an isotropic vector. To simplify the notations and the computations hereafter, assume that $i=1$ and up to rescaling $\epsilon$, that $c$ is primitive integral in $K$ and $c_1(\epsilon)=c_1+\epsilon c$. Let $M$ be the orthogonal of $U=\Z c_1+\Z c$ in $K$. 
Denote the intersection matrix of $U$ by $$\begin{pmatrix}
    0&a\\  a&2b
\end{pmatrix}.$$ then the dual lattice $U^{\vee}$ is generated by $a^{-1}c_1$ and $a^{-1}\widetilde c$ where $\widetilde c=c-2\frac{b}{a} c_1$. The lattice $M$ maps injectively to a finite index sub-lattice of $M_1=c_1^{\bot}/c_1$ and we have the chain of inclusions 
\[U\oplus M\subseteq K\subseteq K^\vee\subseteq U^{\vee}\oplus M^{\vee}.\]

Hence every element of $K^{\vee}$ has a decomposition $a^{-1}(n_1c_1+n_2\widetilde{c})+ \gamma$ where $\gamma\in M^{\vee}$, $n_1,n_2\in \Z$. We denote by $[n_1,n_2]\in (\Z/a\Z)^2$ the image of $a^{-1}(n_1c_1+n_2\widetilde{c})$ in $U^{\vee}/U$. We have
that $\Theta^{c_1(\epsilon)}_K(\tau)$ equals

\begin{align*}
        &\sqrt{\frac{N_1(\epsilon)}{y}}\uparrow_{U\oplus M}^{K}\Bigg(\sum_{\substack{n_1,n_2\in \Z \\ \gamma\in M^\vee}}\beta\left(\frac{y(n_2+\epsilon n_1)^2}{\epsilon a+\epsilon^2 b}\right) q^{-\tfrac{n_1n_2}{a}+\tfrac{n_2^2 b}{a^2}-Q(\gamma)}\mathfrak{e}_{[ n_1, n_2]}\otimes \mathfrak{e}_{[\gamma]}\Bigg)\\
        &=\sqrt{\frac{N_1(\epsilon)}{y}}\uparrow_{U\oplus M}^{K}\left[\Bigg(\sum_{n_1,n_2\in \Z}\beta\left(\frac{y(n_2+\epsilon n_1)^2}{\epsilon a+\epsilon^2 b}\right)q^{-\tfrac{n_1n_2}{a}+\tfrac{n_2^2 b}{a^2}}\mathfrak{e}_{[n_1,n_2]}\Bigg)\right. \\ &\left. \hspace{240pt}\otimes\left(\sum_{\gamma\in M^{\vee}}q^{-Q(\gamma)}\mathfrak{e}_{[\gamma]}\right)\right].
\end{align*}

Notice that because $\beta(t)\ll e^{-\pi t}$, the first sum above is absolutely convergent for $\epsilon$ small enough and $\frac{N_1(\epsilon)}{\epsilon}=a+b\epsilon$. Furthermore, we can decompose the first sum according to $n_2=0$ or $n_2\neq 0$ as: 
\begin{multline*}
  \sqrt{\frac{N_1(\epsilon)}{y}}\sum_{n_1,n_2\in \Z}\beta\left(\frac{y(n_2+\epsilon n_1)^2}{\epsilon a+\epsilon^2 b}\right)q^{-\tfrac{n_1n_2}{a}+\tfrac{n_2^2 b}{a^2}}\mathfrak{e}_{[n_1,n_2]}\\=\left(\frac{\sqrt{N_1(\epsilon)}}{2\pi\sqrt{y}}\sum_{n_1\in \Z}\int_1^{\infty}u^{-\frac{3}{2}}e^{-\pi y u \tfrac{\epsilon n_1^2}{a+b\epsilon}}du\right) \mathfrak{e}_{[n_1,0]}\\
  +\sqrt{\frac{N_1(\epsilon)}{y}}\sum_{n_1,n_2\in \Z,\, n_2\neq 0}\beta\left(\frac{y(\epsilon n_1+n_2)^2}{a\epsilon+b\epsilon^2}\right)q^{-\tfrac{n_1n_2}{a}+\tfrac{n_2^2 b}{a^2}}\mathfrak{e}_{[n_1,n_2]}.
\end{multline*}
Write $\mathfrak{e}_{[n_1,0]}=\mathfrak{e}_{[n_1]}$. The first sum can be rewritten using dominated convergence and Poisson's summation formula as follows:
\begin{align*}
\sum_{n_1\in \Z}\int_1^{\infty}u^{-\frac{3}{2}}e^{-\pi y u \tfrac{\epsilon n_1^2}{a+b\epsilon}}du \mathfrak{e}_{[n_1]} &=\sum_{r\in\Z/a\Z}\sum_{n_1\in r+a\Z}\int_1^{\infty}u^{-\frac{3}{2}}e^{-\pi y u \tfrac{\epsilon n_1^2}{a+b\epsilon}}du \mathfrak{e}_{[r]}\\ &=\sum_{r\in\Z/a\Z}\int_1^{\infty}u^{-\frac{3}{2}}\left(\sum_{n_1\in r+a\Z}e^{-\pi y u \tfrac{\epsilon n_1^2}{a+b\epsilon}}\right)du \mathfrak{e}_{[r]}\\
&=\sum_{r\in\Z/a\Z}\int_1^{\infty}u^{-\frac{3}{2}}\frac{\sqrt{a+\epsilon b}}{a \sqrt{y u\epsilon}}\sum_{m\in \Z}e^{\tfrac{2\pi i mr}{a}-\tfrac{\pi m^2(a+b\epsilon)}{a^2y u\epsilon}}du\mathfrak{e}_{[r]}~.
\end{align*}
The last sum, multiplied by $\frac{1}{2\pi}\sqrt{\frac{N_1(\epsilon)}{y}}$, converges uniformly as $\epsilon\rightarrow 0$ to its constant term which is 
\[\sum_{r\in \Z/a\Z}\frac{1}{2\pi y}\int_{1}^{\infty}u^{-{2}}\mathfrak{e}_{[r]}=\frac{1}{2\pi y}\sum_{r\in \Z/a\Z}\mathfrak{e}_{[r]}=\frac{1}{2\pi y}\uparrow_\Z^{a\Z} (1)~.\]   
As for the second sum, we can write it using the estimate $\beta(t)\ll e^{-\pi t}$ as: 
\begin{multline*}
    \left|\left|\sqrt{\frac{N_1(\epsilon)}{y}}\sum_{n_1,n_2\in \Z, n_2\neq 0}\beta\left(\frac{y(n_2+\epsilon n_1)^2}{\epsilon a+b\epsilon ^2}\right)q^{-\tfrac{n_1n_2}{a}+\tfrac{n_2^2b}{a^2}}\mathfrak{e}_{[n_1,n_2]}\right|\right|\\ \ll \sqrt{\frac{N_1(\epsilon)}{y}}\sum_{n_1,n_2\in \Z, n_2\neq 0} e^{-\pi y\left( n_2^2\left(\tfrac{2b}{a^2}+\tfrac{1}{\epsilon(a+b\epsilon)}\right)-\tfrac{2b\epsilon n_1n_2}{a(a+b\epsilon)} +\epsilon \tfrac{n_1^2}{a+b\epsilon}\right)}~.
\end{multline*}
By computing the discriminant of the quadratic form in the exponent above, we see that there exists $A_1$ such that for $\epsilon \leq A_1$, the following inequality holds for every $n_1,n_2\in \Z$: 
\[n_2^2\left(\frac{2b}{a^2}+\frac{1}{\epsilon a+b\epsilon}\right)-\frac{2b\epsilon n_1n_2}{a(a+b\epsilon)} +\epsilon \frac{n_1^2}{a+b\epsilon}\geq \frac{1}{2a}\left(\frac{n_2^2}{\epsilon} +\epsilon n_1^2\right).\]
Hence combining the previous inequalities and Poisson summation formula we get: 
\begin{multline*}
    \left|\left|\sqrt{\frac{N_1(\epsilon)}{y}}\sum_{n_1,n_2\in \Z, n_2\neq 0}\beta\left(\frac{y(n_2+\epsilon n_1)^2}{\epsilon a+b\epsilon ^2}\right)q^{-\tfrac{n_1n_2}{a}+\tfrac{n_2^2b}{a^2}} \mathfrak{e}_{[n_2]}\right|\right|\\\ll \sqrt{\frac{N_1(\epsilon)}{y}}\sum_{n_1,n_2\in \Z, n_2\neq 0} e^{-\tfrac{\pi y}{2a} \left(\tfrac{n_2^2}{\epsilon} +\epsilon n_1^2\right)}\\
\ll \left(\sqrt{\frac{\epsilon}{y}}\sum_{n_1\in\Z} e^{-\tfrac{\pi y}{2a} \epsilon n_1^2}\right) \left(\sum_{n_2\in \Z, n_2\neq 0} e^{-\tfrac{\pi y}{2a} \tfrac{n_2^2}{\epsilon}} \right)\\
\ll \left(\frac{\sqrt{2aN_1(\epsilon)}}{y\sqrt{\epsilon}}\sum_{n_1\in \Z} e^{-2\pi a \tfrac{n_1^2}{y \epsilon }}\right) \left(\sum_{n_2\in \Z, n_2\neq 0} e^{-\tfrac{\pi y}{2a} \tfrac{n_2^2}{\epsilon}} \right),\\
\end{multline*}
and by normal convergence, the last product converges to $0$ as 
$\epsilon\rightarrow 0$.
We conclude that the limit of $\Theta^{c_1(\epsilon)}_K$ as $\epsilon\rightarrow 0$ is equal to \[\frac{1}{2\pi y}\uparrow_{U\oplus M}^{K}\left(\sum_{\underset{r\in \Z/a\Z}{\gamma\in M^{\vee}}}q^{-Q(\gamma)}\mathfrak{e}_{[\gamma+\frac{rc_1}{a}]}\right)=\frac{1}{2\pi y}\uparrow_{M_1}^{K}(\Theta_{M_1})\] where $M_1=c_1^{\bot}/c_1$, which is the non-holomorphic part of the mixed mock modular form $-\frac{1}{6}E_2\cdot \uparrow_{M_1}^K(\Theta_{M_1})$.


Notice that \[\widehat{\Theta}_\epsilon-\sum_{Q(c_i)>0}a_i\uparrow_{c_i^\perp\oplus \Z c_i}^{K}(\Theta_{c_i^{\bot}}\otimes F_{N_i})+\sum_{Q(c_i)=0}a_i\uparrow_{M_1}^K(\Theta_{M_1})\otimes \tfrac{1}{6}E_2^*\]
transforms like a holomorphic modular form of weight $1+\frac{n}{2}$ with respect to $\overline{\rho}_K$. Hence by letting $\epsilon\rightarrow 0$, it still transforms appropriately, and the previous computations show that its non-holomorphic part vanishes, whereas its holomorphic part is equal to 
\[\Theta^+-\sum_{Q(c_i)>0}a_i\uparrow_{K_{c_i}\oplus \Z c_i}^{K}(\Theta_{c_i^{\bot}}\otimes F^+_{N_i})+\sum_{Q(c_i)=0}\frac{a_i}{6}E_2 \uparrow_{M_i}^K(\Theta_{M_i})\].
This concludes the proof of \Cref{typeII+typeIII}.
\end{proof}

\section{Orthogonal Shimura varieties}
In this section, we give an explicit description of the special divisors in toroidal compactifications of orthogonal Shimura varieties and we derive the main intersection formula, see \Cref{t:intersection-formula}, of special divisors with curves contained in the boundary. For general references on toroidal compactifications, the reader may consult \cite{mumford-amrt} or \cite{bruinierzemel,fiori,looijenga,zemel-parabolic} for the orthogonal case. 

\subsection{Baily-Borel compactification}
Let $(L,\cdot)$ be an even lattice of signature $(2,n)$, let $L^\vee$ be its dual lattice and let $\cD$ be the associated Hermitian symmetric domain, a connected component of $$\{\C x\in \mathbb{P}(L_\C)\,\big{|}\,x\cdot x=0, \,x\cdot\overline{x}>0\}.$$ It embeds into its compact dual $$\cD^{c}:=\{\C x\in \mathbb{P}(L_\C)\,\big{|}\,x\cdot x=0\}$$ which is a quadric hypersurface in the projective space $\mathbb{P}(L_\C)\simeq \mathbb{P}^{n+1}$. 

Let $\Gamma\subset \mathrm{O}^+(L)$ be a subgroup of the stable orthogonal group, defined as the kernel of the reduction $\mathrm{O}^+(L)\rightarrow \mathrm{O}(L^{\vee}/L)$. Then $\Gamma$ acts on $\cD$ and the quotient has the structure of a complex orbifold, it is an \emph{orthogonal Shimura variety}. The restriction of the tautological line bundle $\mathcal{O}_{\mathbb{P}^{n+1}}(-1)$ to $\cD$ is naturally $\Gamma$-equivariant and hence descends to an orbi-line bundle $\mathcal{L}$ on $X_\Gamma:=\Gamma\backslash \cD$, called the \emph{Hodge bundle}. 
For any $\beta\in L^{\vee}/L$ and $m\in -Q(\beta)+\Z$, we define the special divisor of discriminant $(\beta,m)$ in $X_\Gamma$ to be: 
\begin{align}\label{special-divisor}
Z^o(\beta,m)=\bigcup_{\substack{\lambda\in\beta+L \\  Q(\lambda)=-m}}\{ x\in \cD\,\big{|}\,x\cdot \lambda=0\}\textrm{ mod }\Gamma. 
\end{align}

By the Baily-Borel theorem \cite{bailyborel}, the line bundle $\mathcal{L}$ is ample and sections of its powers determine an embedding of the quotient $X_\Gamma$ into a projective space. 
The closure of the image is a projective variety, and is by definition the \emph{Baily-Borel compactification}: $X_\Gamma\hookrightarrow X_\Gamma^{\rm BB}$.

This compactification can be described more explicitly using the rational boundary components of $\cD$ in its topological closure $\overline{\cD}\subset \cD^c$. 

\begin{definition} The {\it rational boundary components} $\cD_I$ of $\cD$ are
in bijection with primitive isotropic sublattices $I\subset L$ as follows: $$\cD_I:=\{\C x\in \overline{\cD}\subset \cD^c\,\big{|}\,{\rm span}\{{\rm Re}\,x,{\rm Im}\,x\}=I_\R\}.$$ \end{definition}

\begin{notation} The possible ranks of an isotropic sublattice of $L$
are $1$ and $2$. To distinguish them,
we henceforth use the letter $I$
when the rank is $1$ and the letter
$J$ when the rank is $2$. \end{notation}

\begin{definition}  For a rank $1$ isotropic lattice, $\cD_I$ is a point. For a rank $2$ isotropic lattice,
$\cD_J$ is a copy of the upper half-plane $\mathbb{H}$.
We call these boundary components
{\it Type III} and {\it Type II} respectively. \end{definition}
 
\begin{remarque} The terminology of Type II and III comes from
the classification \cite{Kulikov, Persson} of one-parameter degenerations of K3 surfaces.\end{remarque}

One defines the {\it rational closure} of $\cD$ to be $$\cD^+:=\cD\cup_J \cD_J\cup_I \cD_I\subset \overline{\cD},$$ endowed with a horoball topology near the boundary components $\cD_I$ and $\cD_J$. An explicit topological description of the Baily-Borel compactification is then
$X_\Gamma^{\rm BB}=(\Gamma\backslash \cD)^{\rm BB}=\Gamma\backslash \cD^+$. Thus,
boundary components of $(\Gamma\backslash \cD)^{\rm BB}$ are the
\begin{enumerate}
    \item Type III components, which are points, in bijection
    with $\Gamma$-orbits of isotropic lines $I\subset L$, and
    \item Type II components, which are modular curves ${\rm Stab}_\Gamma(J)\backslash \mathbb{H}$, in bijection with $\Gamma$-orbits of isotropic planes $J\subset L$.
\end{enumerate}

The closures of Type II modular curves in $(\Gamma\backslash \cD)^{\rm BB}$
contain Type III boundary points, which form the cusps of the modular curves. Such incidences correspond to $\Gamma$-orbits of
inclusions $I\subset J\subset L$.

\subsection{Special divisors at the Type III boundary}

We now describe the geometry of special divisors of $\cD$ near a cusp. Let $I=\Z \delta\subset L$ be a primitive isotropic
sublattice. Given any point $\C x\in \cD$, there is a unique
representative $x\in L_\C$ for which $x\cdot \delta=1$---we cannot have $x\cdot \delta=0$ for signature reasons. This defines an embedding of $\cD$ into the affine hyperplane $\{x\in L_\C\,\big{|}\,x\cdot \delta=1\}$.

Let $U_I\subset {\rm Stab}_\Gamma(I)$ be the unipotent radical of the integral parabolic stabilizer of $I$. Consider an integral (but non-orthogonal) decomposition $L^\vee = \Z\delta'+I^\perp_{L^\vee}$ with $\delta'\cdot \delta=1$.
Then we can write
$$x=\delta'+x_0+c\delta$$ where $x_0$ lies in
$\{\delta,\delta'\}^\perp_\C\simeq (I^\perp/I)_\C$
and $c\in \C$ is uniquely determined by the property that
$x \cdot x=0$. Let $K:=I^\perp/I$. The action of some $\gamma\in U_I$ necessarily takes the form: \begin{align*}
\delta&\mapsto \delta \\ 
v&\mapsto v-(v\cdot w)\delta\textrm{ for }v\in I^\perp \\
\delta'& \mapsto \delta'+ w - (\delta'\cdot w)\delta - \tfrac{1}{2}(w\cdot w)\delta\textrm{ for some }w\in K.
\end{align*} 
This formula above doesn't depend on how $w$ is lifted into $I^\perp$. 

On the coordinate $x_0\in K_\C$,
the unipotent radical $U_I$ acts by a finite index subgroup
of the translation group $K$. In the full stable orthogonal group, we get the full translation group $K$.
Thus, $$U_I \backslash \cD\hookrightarrow U_I \backslash (U_I\otimes \C) \simeq (\C^*)^n$$ admits a torus embedding. This torus
$U_I\otimes \C^*=U_I\backslash (U_I)_\C$ is isogenous to the torus $K\backslash K_\C = K\otimes \C^*$.

We analyze the special divisors in the above torus embedding. It suffices to restrict to $U_I=K$---in the more general case where $U_I\subsetneq K$, we just take the inverse images of the special divisors under the isogeny $U_I\backslash K_\C\to K\backslash K_\C.$
Consider $$Z^o(\widetilde{\lambda}):=\{x\in \cD\,\big{|}\,x \cdot \widetilde{\lambda}=0\}\textrm{ mod }\Gamma$$ for some primitive $\widetilde{\lambda}\in L^\vee$ which passes through the cusp of $(\Gamma\backslash \cD)^{\rm BB}$ corresponding to $I$. Choose a representative of the $\Gamma$-orbit for which $\widetilde{\lambda}\in I^\perp_{L^\vee}$.
Let $\lambda$ be the image of $\widetilde{\lambda}$ under reduction to $I^\perp_{L^\vee}/I^{\rm sat}_{L^\vee}=K^\vee$, where $I^{\rm sat}_{L^\vee}$ is the saturation of $I$ in $L^\vee$.

Note that the vector $\lambda\in K^\vee$ may not be primitive. 
Write $\lambda = u\lambda^{\rm prim}$ where $\lambda^{\rm prim}\in K^\vee$ is primitive. The lifts of $\lambda$ to $\widetilde{\lambda}\in I^\perp_{L^\vee}$ are the set $$\{u\widetilde{\lambda}^{\rm prim}+\tfrac{k}{d}\delta\,\big{|}\,k\in \Z\}$$ where $d$ is the imprimitivity of $\delta\in L^\vee$ and $\widetilde{\lambda}^{\rm prim}$ is a lift of $\lambda^{\rm prim}$. Such a lift is primitive in $L^\vee$ if and only if $\gcd(k,u)=1$. Furthermore,
there exists some $\gamma\in K$ for which $\gamma(\widetilde{\lambda}^{\rm prim})=\widetilde{\lambda}^{\rm prim}+\delta$. So a collection of coset representatives of the $K$-action on $\{u\widetilde{\lambda}^{\rm prim}+\tfrac{k}{d}\delta\,\big{|}\,k\in \Z\}$ is provided by taking all $0\leq k<du$.

The equation of $Z^o(\widetilde{\lambda})$ in coordinates $x_0\in K_\C$ is the affine hyperplane $$x_0\cdot \lambda = - \delta'\cdot \widetilde{\lambda}\in \Q.$$
Under the exponentiation isomorphism $K\backslash K_\C\to K\otimes \C^*$, the equation of $Z^o(\widetilde{\lambda})$ becomes a torsion translate of a character hypersurface $$\chi_{\lambda^{\rm prim}}(z)=z^{\lambda^{\rm prim}} ={\rm exp}(-\tfrac{2\pi ik}{du}-2\pi i \delta'\cdot \widetilde{\lambda}^{\rm prim})$$ where $z\in (\C^*)^n$ is the coordinate-wise exponential of $x_0\in K_\C$. 

 \begin{proposition}\label{typeiii-special} Let $\lambda\in K^\vee$ have imprimitivity $u$. In the torus embedding $U_I\backslash \cD\hookrightarrow U_I\otimes \C^*$ there are $d\varphi(u)$ special divisors associated to lifts of $\lambda$ to a primitive vector $\widetilde{\lambda}\in I^\perp_{L^\vee}$, each one being a torsion translate of the hypersurface $\chi_{\lambda^{\rm prim}}(z)=1$. \end{proposition}

 \begin{proof} When $U_I=K$, the proposition follows from the above computation of the $K$-coset representatives of lifts of $\lambda$. When $U_I\subsetneq K$, the hypersurface $\chi_{\lambda^{\rm prim}}(z)=1$ may break into multiple character hypersurfaces in $U_I\otimes \C^*$---this occurs exactly when $\lambda^{\rm prim}\in K^\vee$ becomes imprimitive in the character group $U_I^\vee\supset K^\vee$. But the equations of the special divisors are still pulled back from $K\otimes \C^*$ along the isogeny $U_I\otimes \C^*\to K\otimes \C^*$, reducing the proposition to the $U_I=K$ case. \end{proof}

 \begin{corollary}\label{typeiii-special-cor} Let $Z_I^o(\beta,m)\subset Z^o(\beta,m)$ be the union of the components of $Z^o(\beta,m)$, see \ref{special-divisor}, whose Zariski closure in $(\Gamma\backslash \cD)^{
 \rm BB}$ contains the Type III cusp associated to $I$. The pullback of $Z_I^o(\beta,m)$ to $U_I\backslash \cD\subset U_I\otimes \C^*$ is $$\bigcup_{\substack{\lambda\in p^K_L(\beta)+K \\  Q(\lambda)=-m}} \{u\textrm{ translates of }\chi_{\lambda^{\rm prim}}(z)=1\}$$ where $u=u(\lambda)$ is the imprimitivity of $\lambda$.\end{corollary}
 
 \begin{proof} 
Note that $Z^o(\beta,m)$ ranges over both primitive and imprimitive vectors. So it contains the divisors $Z^o(\widetilde{\mu})$ associated to primitive lifts of $u'\lambda^{\rm prim}$ ranging over the divisors $u'\mid u$. By the formula $\sum_{u'\mid u} \varphi(u')=u$ and Proposition \ref{typeiii-special}, there are $du$ special divisors associated to lifts of $\lambda$, whose classes may vary in the discriminant group $L^\vee/L$.
Observe that the classes $[\widetilde{\lambda}]$ of lifts of $\lambda$ equidistribute in the classes $\beta\in L^\vee/L$ which push to the class $[\lambda]=p^K_L(\beta)\in K^\vee/K$. Hence there are exactly $u$ translates of the hypersurface $\chi_{\lambda^{\rm prim}}(z)=1$ appearing in $Z_I^o(\beta,m)$. \end{proof}

\subsection{Toroidal compactification}\label{s:tor-comp} We now describe the toroidal
compactification in a neighborhood of the Type III cusps. As above, let $I\subset L$ denote an isotropic line, and let $K:=I^\perp/I$. Define a finite index subgroup $\Gamma_I\subset O^+(K)$ by the exact sequence $$0\to U_I\to {\rm Stab}_\Gamma(I)\to \Gamma_I\to 0.$$ Let $C_K\subset K_\R$ denote the {\it positive cone}---a connected component of the positive norm vectors in $K_\R\simeq \R^{1,n-1}$ and let $C_K^+$ denote its {\it rational closure}: The union of $C_K$ with all rational isotropic rays at its boundary.

\begin{definition} A {\it $\Gamma$-admissible collection of fans} (or simply {\it fan}) is a collection $\Sigma=\{\Sigma_I\}$ of polyhedral decompositions $\Sigma_I=\{\sigma\}$ of $C_K^+$ ranging over all primitive isotropic $I\subset L$, such that: 
\begin{enumerate}
\item the cones $\sigma\subset C_K^+$ are strictly convex, rational polyhedral cones, and $\Sigma_I$ is closed under taking intersections and faces, and
\item the collection $\Sigma=\{\Sigma_I\}$ is $\Gamma$-invariant with only finitely many $\Gamma$-orbits of cones $\sigma\in \Sigma_I$.
\end{enumerate}
\end{definition}

Equivalently, we can choose one polyhedral decomposition $\Sigma_I$ for each $\Gamma$-orbit of isotropic line $I$, and require that it be $\Gamma_I$-invariant, with finitely many orbits of cones. A useful visualization is to consider $\{\mathbb{P}\sigma\}_{\sigma\in\Sigma_I}$ which is a $\Gamma_I$-periodic tiling of (rationally cusped) hyperbolic $(n-1)$-space $\mathbb{P} C_K^+$.

We extend the algebraic torus $U_I\otimes \C^*$ by the 
infinite-type toric variety $X(\Sigma_I)$ whose fan is $\Sigma_I$.
It admits an action of $U_I\otimes \C^*$ whose orbits are in bijection
with the cones $\sigma\in \Sigma_I$. The orbit corresponding to $\sigma$ is isomorphic to $(U_I/\textrm{span}(\sigma))\otimes \C^*\simeq (\C^*)^{\rm codim(\sigma)}$.

Let $V_I$ be a $\Gamma_I$-invariant, analytic tubular neighborhood
of the {\it Type III strata} of $X(\Sigma_I)$:
These are the torus orbits corresponding to cones $\sigma$
satisfying $\sigma\cap C_K\neq \emptyset$---exactly the 
strata corresponding to cones $\sigma$ which are neither
the origin $0\in K$ nor a rational isotropic ray.
Such cones $\sigma$ only have finite
stabilizers in $\Gamma_I$, so we may assume that the action
of $\Gamma_I$ on $V_I$ is properly discontinuous. Let $$A_I:=V_I\cap (U_I\otimes \C^*)$$ be the intersection with the open torus orbit.

Choosing $V_I$ sufficiently small, we can ensure that
$\Gamma_I\backslash A_I$ embeds into $\Gamma\backslash \cD$. This follows from the fact that a punctured neighborhood of the Type III boundary component associated to $I$ is locally modeled as the quotient of some analytic open subset of $\cD$ by the parabolic stabilizer ${\rm Stab}_\Gamma(I)$.
Then, we define the {\it Type III extension} of $\Gamma\backslash \cD$ to be the glued analytic space $$(\Gamma\backslash \cD)^{\Sigma,{\rm III}}:=(\Gamma\backslash \mathbb{D})\cup_{\Gamma_I\backslash A_I} (\Gamma_I\backslash V_I)$$ with the union taken over all $I$.

Taking the further union with the Type II extensions $\Gamma_J\backslash V_J$ defined in the following subsection
gives the {\it toroidal compactification}: $$(\Gamma\backslash \cD)^\Sigma:=(\Gamma\backslash \cD)^{\Sigma,{\rm III}}\cup_{\Gamma_J\backslash A_J}(\Gamma_J\backslash V_J).$$

For any $\Sigma$, it admits a morphism $(\Gamma\backslash \cD)^\Sigma\to (\Gamma\backslash \cD)^{\rm BB}.$

\begin{definition}
We say that $\Gamma$ is {\it neat} if the eigenvalues of every element $\gamma\in \Gamma$ acting on $L_\C$ generate a torsion free subgroup of $\C^*$. 
\end{definition}

\begin{definition}
We say that $\Sigma$ is {\it regular} if every cone
$\sigma$ is a standard affine cone ($\sigma\cap U_I$ is isomorphic
to $\N^{\dim(\sigma)}$ as a monoid). More generally, $\Sigma$ is {\it simplicial} if the generators of any
cone $\sigma\in \Sigma_I$ are linearly independent.
\end{definition}

If $\Gamma$ is neat and $\Sigma$ is regular,
then the toroidal compactification $X_\Gamma^\Sigma= (\Gamma\backslash \cD)^{\Sigma}$ is smooth with reduced normal crossings boundary, because $\Gamma$, $\Gamma_I$, $\Gamma_J$ will act freely on $\cD$ and the respective toroidal extensions. Since the toric variety associated to a simplicial cone $\sigma$ has only finite quotient singularities, it is a local $\Q$-Poincar\'e duality space. For the same reason, $X^\Sigma_\Gamma$ is a $\Q$-Poincar\'e duality space (regardless of whether $\Gamma$ is neat) for any simplicial $\Sigma$.

\begin{remarque} \label{r:simplicial}
We can and will assume that $\Gamma$ is neat
and $\Sigma$ is regular. One can reduce the
main result of the paper for simplicial fans $\Sigma$ 
and arbitrary $\Gamma$ to this case:
By passing to a finite index subgroup
$\Gamma'\subset \Gamma$ we can make the action
neat, and by refining $\Sigma\prec \Sigma'$,
we can make the fan regular. We get a
morphism $$\pi\colon X^{\Sigma'}_{\Gamma'}\to X_{\Gamma'}^{\Sigma}\to X_\Gamma^{\Sigma}$$ which is the composition of a birational
morphism and a finite cover.
Then $\pi^*\colon H^2(X_\Gamma^\Sigma,\Q)\to H^2(X_{\Gamma'}^{\Sigma'},\Q)$
is an embedding.

The main theorem is proved by intersecting
the Zariski closures $Z(\beta,m)=\overline{Z^o(\beta,m)}$ with curve
classes $C\in H_2(X_\Gamma^\Sigma,\Q)$. Such a class $C$ can be expressed as $[\Gamma:\Gamma']C \equiv \pi_*D$ for some $D\in H_2(X_{\Gamma'}^{\Sigma'},\Z)$ which intersects the exceptional divisors of $\pi$ to be zero: This is achieved by using a wrong-way map $D:=({\rm PD}\circ \pi^*\circ {\rm PD})(C).$ By push-pull, \begin{align*} Z(\beta,m)\cdot_{X_\Gamma^\Sigma} C &= \frac{1}{[\Gamma:\Gamma']} \pi^*Z(\beta,m)\cdot_{X_{\Gamma'}^{\Sigma'}} D \\ &=\frac{1}{[\Gamma:\Gamma']}
Z(\beta,m)\cdot_{X_{\Gamma'}^{\Sigma'}}D.\end{align*} The last
equality holds because $D$ intersects the exceptional divisors
to be zero and the strict transforms of the $Z(\beta,m)$'s for $\Gamma,
\Sigma$ are the $Z(\beta,m)$'s for $\Gamma',\Sigma'$. The same equality holds for any boundary divisor in $X_\Gamma^\Sigma$ associated to a ray in $\Sigma$. This implies that the pairing with $D$ of the generating series from Theorem \ref{t:precise-main} applied  to $X_{\Gamma'}^{\Sigma'}$ is equal (up to the factor $[\Gamma:\Gamma']$) to the pairing of $C$ with the analogous series for $X_{\Gamma}^{\Sigma}$. This proves that the validity of Theorem \ref{t:precise-main} for $X_{\Gamma'}^{\Sigma'}$ implies its validity for $X_{\Gamma}^{\Sigma}$.
\end{remarque}

\subsection{The Type II boundary}\label{subsec-ii} We now roughly
sketch the structure of the Type II boundary component associated to the $\Gamma$-orbit of an isotropic plane $J\subset L$. 

Choose a basis $J=\Z\delta\oplus \Z\lambda$
and as before, represent $\C x\in \cD$ uniquely by the $x\in L_\C$ for which $x\cdot \delta=1$. Choose $\delta',\lambda'$
which generate a lattice $J'=\Z\delta'+\Z\lambda'$ in $L^\vee$ realizing
${\rm Hom}(J,\Z)$ and write $$x= (\delta'+\tau\lambda')+x_0+(c_1\delta+c_2\lambda)$$
where $x_0\in \{\delta,\delta',\lambda,\lambda'\}^\perp_\C\cong (J^\perp/J)_\C$ and $c_1,c_2\in \C$. The condition $x\cdot x=0$ specifies that $(c_1,c_2)\in L\subset \C^2$ lies in an affine line $L$, and the condition $x\cdot \overline{x}>0$ further implies that $(c_1,c_2)$ lies in a half-space $\mathbb{H}_{(x_0,\tau)}\subset L$ inside this line depending on $(x_0,\tau)\in (J^\perp/J)_\C\times \mathbb{H}$. The unipotent radical $U_J\subset {\rm Stab}_\Gamma(J)$
is an extension of the form $$0\to \Z \to U_J\to T_J\to 0 $$
where $T_J$ is a finite index translation subgroup $$T_J\subset (J^\perp/J)\otimes (\Z\oplus \Z\tau)\subset (J^\perp/J)_\C.$$

First taking the quotient by $\Z\subset U_J$ embeds
$\Z\backslash \cD$ as a holomorphic punctured disk bundle
(the punctured disks are $\Z\backslash \mathbb{H}_{(x_0,\tau)}$) over the product $(J^\perp/J)_\C\times \mathbb{H}$. We then
quotient by the action of $T_J$. The punctured disk bundle descends to one over a family of abelian varieties $T_J\otimes \mathcal{E}$. Here $\mathcal{E} \to \mathbb{H}$ is the universal elliptic curve whose fiber of $\tau\in \mathbb{H}$ is $\C/(\Z\oplus \Z\tau)$.

We define $V_J$ as a tubular neighborhood of the zero section of the extension of $U_J\backslash \cD$ from a punctured disk bundle to a disk bundle. Then the {\it Type II toroidal extension} $(\Gamma\backslash \cD)^{\rm II}$ is the result of gluing $\Gamma\backslash \cD$ to $\Gamma_J\backslash V_J$ along a punctured tubular neighborhood $\Gamma_J\backslash A_J$ of the zero section. Here $\Gamma_J$ is defined by the exact sequence $$0\to U_J\to {\rm Stab}_\Gamma(J)\to \Gamma_J\to 0$$ and is a finite index subgroup $\Gamma_J\subset {\rm SL}(J)\times O(J^\perp/J).$ The Type II extension is notably independent of the choice of a fan $\Sigma$; it is the same in every toroidal compactification,
and the boundary divisor is always isomorphic to $\Gamma_J\backslash (T_J\otimes \mathcal{E})$.

Finally, we comment on how the Type II and III locus meet. Given a flag of isotropic subspaces $I\subset J \subset L$,
the quotient $\Gamma_I\backslash V_I$ contains a tubular neighborhood of the intersection of the closure of the Type II locus $\Gamma_J\backslash V_J$ with the Type III locus. This intersection is a maximal degeneration of the family $T_J\otimes \mathcal{E}$ of abelian varieties over the cusp of the modular curve ${\rm Stab}_\Gamma(J)\backslash \cD_J$ corresponding to $I\subset J$. More explicitly, it is
a Mumford degeneration, whose fan is the quotient fan
associated to the isotropic ray $\overline{J}\subset I^\perp/I$, necessarily a cone of $\Sigma_I$. This fan
is the cone over a periodic polyhedral decomposition of $(J^\perp/J)_\R$ (put at height $1$), which is invariant under the action of ${\rm Stab}_{\overline{J}}(\Gamma_I)$.

A special divisor $Z(\widetilde{\lambda})$ intersects the Type II extension associated to $J$ if and only if we can represent $\widetilde{\lambda}\in J^\perp_{L^\vee}.$ Let $M:=J^\perp/J$. In general, if we denote by $J^{\rm sat}_{L^\vee}$ the saturation of $J$ in $L^\vee$, then the image $\lambda\in J^\perp_{L^\vee}/J^{\rm sat}_{L^\vee}=M^\vee$ may be imprimitive, and we can write $\lambda = u\lambda^{\rm prim}.$ Essentially, the computations
are the same as in the Type III case: The primitive lifts of $\lambda$ correspond to primitive $u$-torsion translates of a character hypersurface $\{\chi_{\lambda^{\rm prim}}=1\}\subset T_J\otimes \mathcal{E}$,
 which is a family of codimension $1$ abelian subvarieties of $T_J\otimes \mathcal{E}$.

\begin{proposition} In any toroidal compactification, the intersection of $Z(\beta,m)$ with the boundary divisor $\Delta_J$ associated to $J$ is $$\bigcup_{\substack{\lambda\in p_L^M(\beta)+M \\ Q(\lambda)=-m}} \{u^2\textrm{ translates of 
 }\chi_{\lambda^{\rm prim}}(z)=1\}$$ \end{proposition}

 There are $u^2$ translates rather than $u$ translates
 of the character hypersurface, as the $u$-torsion on an elliptic curve has size $u^2$.
\subsection{Intersection formula} Let $\sigma = \textrm{span}\{c_1,\dots,c_{n-1}\}\in \Sigma_I$ be a codimension $1$ Type III face of a regular fan
and let \begin{align*} \sigma^+ & =\textrm{span}(\sigma,c^+) \\
 \sigma^- & =\textrm{span}(\sigma,c^-) \end{align*}
 be the two maximal cones containing $\sigma$, so that $\sigma^+\cap \sigma^-=\sigma$.
The corresponding open torus orbit is isomorphic to $\C^*$ and the closed stratum is $C_\sigma=\mathbb{P}^1$. A small analytic neighborhood of
$C_\sigma \subset (\Gamma\backslash \cD)^{\Sigma}$ is analytically-locally modeled by the toric variety with two maximal cones $\sigma^+,\sigma^-$, which is isomorphic to \begin{align*} {\rm Spec}\,\C[(\sigma^+)^\vee]\cup {\rm Spec}\,\C[(\sigma^-)^\vee] & \simeq \C^n\cup_{\C^{n-1}\times \C^*}\C^n \\ &\simeq{\rm Tot}(\mathcal{O}_{\mathbb{P}^1}(a_1)\oplus \cdots \oplus \mathcal{O}_{\mathbb{P}^1}(a_{n-1}))\end{align*} for some
integers $a_i\in \Z$.

Since $\sigma^+, \sigma^-$ are standard affine cones, $\det(c_1,\dots,c_{n-1},c^+)=\pm 1$ and
$\det(c_1,\dots,c_{n-1},c^-)=\mp 1$. Thus, there is an expression $$c^++c^- + \sum_{i=1}^{n-1} a_i c_i=0$$ for unique integers $a_i\in \Z$, identical to the integers as above.

\begin{proposition}\label{p:intersection} Let $\lambda$ be primitive in $K^\vee$ and consider $Z(\lambda)$ the Zariski closure of the character hypersurface $$Z^o(\lambda):=\{z\in K\otimes \C^*\,\big{|}\,\chi_\lambda(z)=1\}$$ in the toric variety $X(\sigma^+,\sigma^-)$, with $\sigma, \sigma^{\pm}$ as above. Then $$Z(\lambda)\cdot C_\sigma = \tfrac{1}{2}\left(\textstyle |\lambda\cdot c^+|+|\lambda\cdot c^-| + \sum_i a_i |\lambda\cdot c_i|\right).$$
In particular, if $\lambda\cdot c^{\pm}$ and $\lambda\cdot c_i$ all have the same sign, $Z(\lambda)\cdot C_\sigma=0$. \end{proposition}

\begin{proof} Let $X(\mathfrak{F})$ be a finite type toric variety, and let $v_i$ be primitive integral generators of the
one-dimensional rays of $\mathfrak{F}$. There is a bijection $\Delta_{v_i}\leftrightarrow v_i$ with toric boundary divisors.
 We have a linear equivalence $$Z(\lambda) \equiv \sum_i \textrm{max}\{0,\lambda\cdot v_i\}\Delta_{v_i}$$ which can, for instance, be verified by intersecting $Z(\lambda)$ with 
 cocharacters $v_i\otimes \C^*\subset K\otimes \C^*$. Note
 that replacing $\lambda$ with $-\lambda$ does not change
 the answer, since $$\sum_i \textrm{max}\{0,\lambda\cdot v_i\}\Delta_{v_i}-\sum_i \textrm{max}\{0,-\lambda\cdot v_i\}\Delta_{v_i}=\sum_i (\lambda\cdot v_i)\Delta_{v_i}\equiv 0.$$

Applying this formula to the case at hand, we have
\begin{align*} Z(\lambda)\cdot \mathbb{P}^1 =&\,\textrm{max}\{0,\lambda\cdot c^+\}\Delta_{c^+}\cdot \mathbb{P}^1+\textrm{max}\{0,\lambda\cdot c^-\}\Delta_{c^-}\cdot \mathbb{P}^1 \\ &+\sum_i \textrm{max}\{0,\lambda\cdot c_i\}\Delta_{c_i}\cdot \mathbb{P}^1\end{align*}
where $\mathbb{P}^1=C_\sigma$.
Re-expressing the function ${\rm max}\{0,r\}=\tfrac{1}{2}(|r|+r)$ and noting that $\sum_i(\lambda\cdot c_i) \Delta_{c_i}\equiv 0,$
we get $$Z(\lambda)\cdot \mathbb{P}^1 = \tfrac{1}{2}(|\lambda\cdot c^+|(\Delta_{c^+}\cdot \mathbb{P}^1)+|\lambda\cdot c^-|(\Delta_{c^-}\cdot \mathbb{P}^1)+\textstyle \sum_i |\lambda\cdot c_i|(\Delta_{c_i}\cdot \mathbb{P}^1)).$$

We now describe geometrically the toric boundary divisors associated to $c^{\pm}$, $c_i$. In the total space interpretation $$X(\sigma^+,\sigma^-)={\rm Tot}(\mathcal{O}_{\mathbb{P}^1}(a_1)\oplus\cdots\oplus \mathcal{O}_{\mathbb{P}^1}(a_{n-1})),$$ the divisors
$\Delta_{c^+}$ and $\Delta_{c^-}$ are the fibers over $0,\infty\in \mathbb{P}^1$ and so $\Delta_{c^{\pm}}\cdot \mathbb{P}^1=1$.
The divisor $\Delta_{c_i}$ is ${\rm Tot}(\mathcal{O}_{\mathbb{P}^1}(a_1)\oplus\cdots \oplus 0 \oplus \cdots \oplus \mathcal{O}_{\mathbb{P}^1}(a_{n-1}))$, leaving out the $i$th line bundle factor. Hence the normal bundle to $\Delta_{c_i}$ is the pullback of $\mathcal{O}_{\mathbb{P}^1}(a_i)$ to $X(\sigma^+,\sigma^-)$
and so its restriction to the zero section $\mathbb{P}^1$ has degree $a_i$. We conclude $$Z(\lambda)\cdot C_\sigma = \tfrac{1}{2}(|\lambda\cdot c^+|+|\lambda\cdot c^-|+\textstyle \sum_i a_i|\lambda\cdot c_i|).$$
as desired.\end{proof}

\begin{theorem}\label{t:intersection-formula} Let $C_\sigma\subset X_\Gamma^\Sigma$ be a one-dimensional Type III boundary stratum of the toroidal compactification, corresponding to the cone $\sigma=\sigma^+\cap \sigma^-$ in the hyperbolic lattice $K=I^\perp/I$. We have: \begin{align*}
(Z(0,0)\cdot C_\sigma)\mathfrak{e}_0 &+\sum_{\beta\in L^{\vee}/L}\sum_{m\in -Q(\beta)+\Z} (Z(\beta,m)\cdot C_\sigma)\mathfrak{e}_\beta q^m  \\
&=\sum_{\lambda \in K^\vee}  p^+(\lambda)
p_K^L(\mathfrak{e}_{[\lambda]})q^{-Q(\lambda)} 
\end{align*} where $[\lambda]$ is the class
of $\lambda$ in $K^\vee/K$, and \[p^+(\lambda)=\tfrac{1}{2}(|\lambda\cdot c^+|+|\lambda\cdot c^-|+\sum_i a_i|\lambda\cdot c_i|)\] is the piecewise linear function from \Cref{p:intersection}.
\end{theorem}

\begin{proof} Since $C_\sigma$
is contracted in $X_\Gamma^{\rm BB}$, we have $Z(0,0)\cdot C_\sigma=0$ because the Hodge bundle is ample on $X_\Gamma^{\rm BB}$. Next, observe that only the components $Z_I(\beta,m)\subset Z(\beta,m)$ entering the Baily-Borel cusp associated to $I$ can possibly intersect $\mathbb{P}^1=C_\sigma$. We wish to compute the intersection number $$Z(\beta,m)\cdot_{X_\Gamma^\Sigma}C_\sigma=Z_I(\beta,m)\cdot_{\Gamma_I\backslash X(\Sigma_I)}C_\sigma = Z_I(\beta,m)\cdot_{X(\Sigma_I)} C_\sigma.$$
The first equality holds from the construction of the toroidal compactification, while the second equality holds because the action of $\Gamma_I$ freely permutes the infinitely many $\mathbb{P}^1$s in the orbit $\Gamma_I\cdot C_\sigma$. Thus, to compute the intersection number with $Z_I(\beta,m)$ in $\Gamma_I\backslash X(\Sigma_I)$, it suffices to pull back $Z_I(\beta,m)$ to the toric variety $X(\Sigma_I)$, then intersect the pullback with a single
orbit representative $C_\sigma$.

On the toric variety $X(\Sigma_I)$, we have by Corollary \ref{typeiii-special-cor} that
$$Z_I(\beta,m)\sim_{\rm num} \sum_{\substack{\lambda\in p_L^K(\beta)+K \\ Q(\lambda)=-m}} u(\lambda) Z(\lambda^{\rm prim})$$ where $Z(\lambda^{\rm prim})$
is the closure of the character hypersurface and $u(\lambda)$ is the imprimitivity of $\lambda$. While this divisor is not finite, it is locally finite in a neighborhood of the Type III toric boundary.

We take the intersection
with $C_\sigma$ using \Cref{p:intersection}. It gives $$Z_I(\beta,m)\cdot_{X(\Sigma_I)} C_\sigma = \sum_{\substack{\lambda\in p_L^K(\beta)+K \\ Q(\lambda)=-m}} u(\lambda)p^+(\lambda^{\rm prim}),$$
which ends up as a sum with only finitely many nonzero terms.
Finally, we observe that $p^+(r\lambda)=rp^+(\lambda)$ for any $r\in \R^+$ and so $u(\lambda)p^+(\lambda^{\rm prim})=p^+(\lambda)$. The theorem follows. \end{proof}

\section{Vanishing theorems and splitting criteria}
We prove in this section the splitting statement \ref{t:algsplitting} which allows us to prove modularity separately in the boundary and in the interior of $X_\Gamma^{\Sigma}$. We refer to \cite[Lem.~26]{greer} for a similar application. 
\subsection{Vanishing from super-rigidity}\label{s:superR}
The goal of this subsection is to prove that the first Chern class map
\[c_1:\Pic(X_\Gamma^\Sigma )_\Q \to H^{1,1}(X_\Gamma^\Sigma,\Q ) \]
is an isomorphism. This implies that our modularity results for special cycles automatically hold at the level of the $\Q$-Picard group, even though the methods are topological. Using the exponential exact sequence and the Hodge decomposition, it suffices to show that
\[ H^1( X_\Gamma^\Sigma,\Q ) = 0.\]
First, observe that by deforming loops away from the boundary strata that have real codimension 2, we have surjectivity of the pushforward
\[ \pi_1(X_\Gamma) \to \pi_1( X_\Gamma^\Sigma). \]
Since $H_1$ is the abelianization of $\pi_1$, it suffices to prove that the interior $X_\Gamma = \Gamma\backslash \cD$ satisfies
\[H^1(X_\Gamma,\Q)=0.\]
Since $\cD$ is contractible, this coincides with group cohomology of the arithmetic group $\Gamma$. A consequence of Margulis super-rigidity \cite[\S 7.6, Corollary 7.6.17]{margulis}, see also \cite{davewitte}, is the following.

\begin{theorem}
Let $G$ be a semisimple Lie group of rank $r\geq 2$ and $\Gamma\subset G$ an irreducible arithmetic subgroup. Then we have $H^1(\Gamma,V)=0$ for any $\Gamma$-module $V$.
\end{theorem}
Since $O(p,q)$ has rank $\min(p,q)$, the desired vanishing statement holds for any $\Gamma\subset O(2,n)$ for $n>2$ and $V=\Q$ the trivial module. In the case $n=2$, $O(2,2)$ is isogenous to $SL(2)\times SL(2)$, so we need to assume that $\Gamma$ is not commensurable to a product $\Gamma_1\times \Gamma_2$. 

\subsection{Splitting homology classes}
Let $X$ be a smooth complex variety, $\Delta = \bigcup_i \Delta_i\subset X$ a simple normal crossings divisor, and $U = X\setminus \Delta$ its open complement. A basic question is: which rational homology classes $\alpha \in H_2(X,\Q)$ can be expressed as $\alpha = \alpha_0 + \alpha_1$
where $\alpha_0 \in H_2(U,\Q)$ and $\alpha_1 \in H_2(\Delta,\Q)$? In this case, we say that $\alpha$ {\it splits}. Consider the covering of $X$ by two open subsets: 
$$X = U\cup N_\Delta$$ where $N_\Delta$ is a tubular neighborhood of $\Delta$. Then $\Delta$ is a deformation retract of $N_\Delta$. The Mayer-Vietoris sequence for rational homology with this covering reads
\[ H_2(U) \oplus H_2(N_\Delta) \to H_2(X) \overset{\delta}\to H_1(N_\Delta^*) \to H_1(U)\oplus H_1(N_\Delta) \]
where $N_\Delta^*$ denotes $N_\Delta \setminus \Delta$. By exactness, $\alpha \in H_2(X)$ can be split as above if and only if $\delta(\alpha)=0$. Since $H_1(U)$ always vanishes in our desired application, we will use the following criterion whose proof is immediate from the Mayer-Vietoris sequence.
\begin{proposition}\label{splitcriterion}
If the pushforward map $H_1(N_\Delta^*) \to H_1(N_\Delta)$ is injective, then any class $\alpha\in H_2(X)$ splits.
\end{proposition}
In the case where $\Delta$ is irreducible and smooth, this criterion is easily checked using a Gysin sequence.
\begin{proposition}\label{irredcriterion}
If $\Delta$ is irreducible and smooth, then the pushforward map $H_1(N_\Delta^*) \to H_1(N_\Delta)$ is injective if and only if the normal bundle $N_{\Delta/X}$ has non-trivial Euler class.
\end{proposition}
\begin{proof}
Since $\Delta$ is smooth, the punctured neighborhood $N^*_\Delta$ retracts to a circle bundle over $\Delta$. The Gysin sequence for $N^*_\Delta \to \Delta$ reads:
$$H_2(\Delta) \to H_0(\Delta) \to H_1(N^*_\Delta) \to H_1(\Delta) \to 0.$$
The first map in this sequence is given by cap product with $e(N_{\Delta/X})$. Since $H_0(\Delta)\simeq \Q$, this map is surjective if and only if $e(N_{\Delta/X})\neq 0$. By exactness of the Gysin sequence, this in turn is equivalent to injectivity of the map $H_1(N^*_\Delta) \to H_1(\Delta)\simeq H_1(N_\Delta)$.
\end{proof}

In the case where $\Delta$ is reducible, $N_\Delta$ is a plumbing of normal bundles to the irreducible components $\Delta_i$. The splitting criterion can fail to be satisfied even when the individual $\Delta_i$ have non-trivial normal bundles, e.g. for $\Delta$ the triangle of lines in $\P^2$ we get $N^*_\Delta$ deformation retracts to $T^3$ while $H_1(\Delta)\simeq \Q$, so $H_1(N^*_\Delta) \to H_1(\Delta)$ cannot be injective.\\

To check the criterion of Prop. \ref{splitcriterion} in the context of toric boundary divisors, we introduce the moment map. Let $(X,\Delta)$ be a projective toric variety of dimension $n$ defined by a lattice polytope $P$ in the character lattice of $X$. The Hamiltonian action of $T^n\subset (\C^*)^n$ on $X$ defines a surjective map
\[\mu: X \to P\]
whose fiber at an interior point of $P$ is $T^n$, canonically identified with the cocharacter lattice quotient. The fiber at a point in the relative interior of a face of codimension $k$ is $T^{n-k}$. The pre-image of the boundary $\partial P\subset P$ is the toric boundary divisor $\Delta$:
\[ \mu^{-1}(\partial P) = \Delta = \textstyle \bigcup_i \Delta_i \]
with each facet $P_i\subset \partial P$ giving a component $\Delta_i\subset \Delta$. From this description, the punctured neighborhood $N^*_\Delta$ retracts to $\mu^{-1}(S^{n-1})$ which is diffeomorphic to $S^{n-1}\times T^n$.\\

Consider now the Type III boundary divisors $\Delta_I=\cup_c \Delta_{I,c}$ of a regular toroidal compactification $X_\Gamma^\Sigma$ lying over a Type III cusp of $X_\Gamma^{\rm BB}$. Here $\R_{\geq 0}c\in\Sigma_I$ are the one-dimensional cones of the fan $\Sigma_I$ supported on the rational closure of the positive cone $C_K^+\subset K_\R$, $K=I^\perp/I$.

To get a polytope from $\Sigma_I$ requires a convex $\Z$-piecewise linear function $f\colon C_K\to \R$ whose bending locus is $\Sigma_I$ and for which $\gamma^*f-f$ is linear for all $\gamma\in \Gamma_I$. Such a function may not exist, as $X_\Gamma^\Sigma$ need not be projective for some $\Sigma$. But, there is a combinatorial model for what the polytope would be: the dual polyhedral complex $(\Sigma_I)^\vee$ which has a cell $\R_+^{{\rm codim}(\sigma)}$ for each cone $\sigma\in \Sigma_I$. There is still a fibration $$\mu\colon X(\Sigma_I)\to (\Sigma_I)^\vee$$ whose fiber over the relative interior of a codimension $k$ stratum of $(\Sigma_I)^\vee$ is $T^{n-k}$.

Consider an analytic neighborhood of Type III toric boundary in the unipotent quotient $U_I\backslash \mathbb{D}$, called $V_I\subset X(\Sigma_I)$ in \Cref{s:tor-comp}. We have a restricted fibration $$\mu\colon V_I \to P_I\subset (\Sigma_I)^\vee.$$ The action of the group $\Gamma_I={\rm Stab}_\Gamma(I)/U_I$ can be made equivariant with respect to the fibration $\mu$, as it
 acts torically.

The punctured neighborhood $A_I= V_I\cap (U_I\otimes \C^*)$ of the toric boundary fibers over the open cell $O=\R^n_+\cap P_I$ and is thus a trivial fiber bundle $T^n\times O$. The fibers of $\mu$ are, in the toric coordinates $(z_1,\dots,z_n)\in U_I\otimes \C^*$ given by $|z_i|=r_i$ for $r_i\in \R_+$. So the action of $\Gamma_I$ on $T^n\times O$ is, on the $T_n$ factor, induced by the action of $\Gamma_I$ on $I^\perp/I\supset U_I=H_1(T_n,\Z)$.

Let $\Delta= X_\Gamma^\Sigma\setminus X_\Gamma$. The inclusion $N_\Delta^* \to N_\Delta$ is locally modeled by $$\Gamma_I\backslash A_I\to \Gamma_I\backslash V_I$$ in a neighborhood of the cusp $I$. Note that $\Delta_I\subset\Gamma_I\backslash (V_I\setminus A_I)$ and the latter retracts to the former, by $\Gamma_I$-equivariantly retracting the Type II components of $V_I\backslash A_I$ to their intersection with the Type III locus.

Let $\widetilde{\Delta}_I=V_I\setminus A_I$ be the $\Gamma_I$-cover of $\Delta_I$. The deformation retraction $V_I\to \widetilde{\Delta}_I$ can be made $\Gamma_I$-equivariant.
 

\begin{proposition}\label{groupcoho}
 The morphism $N^*_{\Delta_I} \to \Delta_I$ induces an isomorphism on $H_1$. Both can be identified with the group homology $H_1(\Gamma_I,\Q)$.
\end{proposition}
\begin{proof}
The Leray-Cartan spectral sequence computes the homology of an orbifold quotient $X/G$ as follows:
\[ H_p(G, H_q(X,\Q)) \Rightarrow H_{p+q}(X/G,\Q). \]
The $E_2$ page yields an exact sequence for computing $H_1$:
\[H_0(G,H_1(X,\Q)) \to H_1(X/G,\Q) \to H_1(G,H_0(X,\Q)) \to 0. \]
The spectral sequence is functorial for $G$-equivariant morphisms, so applying this to the $\Gamma_I$-maps $A_I \to V_I\to \widetilde{\Delta}_I$ we get a morphism of exact sequences:
\[\xymatrix{
H_0(\Gamma_I,H_1(T^n,\Q)) \ar[r]\ar[d] & H_1(N^*_{\Delta_I},\Q) \ar[r]\ar[d] & H_1(\Gamma_I,\Q) \ar[r]\ar[d] & 0\\
H_0(\Gamma_I,H_1(\widetilde{\Delta}_I,\Q)) \ar[r] & H_1(\Delta_I,\Q) \ar[r] & H_1(\Gamma_I,\Q) \ar[r] & 0\\
}\]
We claim both groups on the far left vanish. First, $H_1(\widetilde{\Delta}_I,\Q)=0$ because $\widetilde{\Delta}_I$ is a union of toric varieties along toric overlaps. The upper left group is $H_0(\Gamma_I,(U_I)_\Q)\simeq H^0(\Gamma_I,K_\Q)^*$, the invariant subspace of the standard representation. As $\Gamma_I\subset O^+(K)$ is finite index, the Borel Density Theorem implies it acts without fixed vectors on $K$.
\end{proof}

\begin{theorem}\label{t:splitting}
Let $X = X_\Gamma^\Sigma$, $U=X_\Gamma$, and $\Delta=X_\Gamma^\Sigma\setminus X_\Gamma$. Then every homology class $\alpha\in H_2(X,\Q)$ splits.
\end{theorem}
\begin{proof}
Let $\Delta_{\rm II}:=\bigcup_J \Delta_J$ and $\Delta_{\rm III}:=\bigcup_I \Delta_I$ be the unions of the Type II and III boundary components, respectively. We have $\Delta=\Delta_{\rm II}\cup \Delta_{\rm III}$.

Assuming $\Sigma$ is taken fine enough, $\Delta_{\rm II}=\coprod_J \Delta_J$ is a disjoint union of smooth divisors, in bijection with orbits of isotropic planes $J\subset L$. On the other hand, the union $\Delta_{\rm III}=\coprod_I \bigcup_c \Delta_{I,c}$ is a disjoint union of SNC divisors, with one connected SNC divisor for each orbit of isotropic line $I\subset L$. The dual complex of this divisor $\bigcup_c \Delta_{I,c}$ is the hyperbolic orbifold $\Gamma_I\backslash \mathbb{P}C_K$. 

Applying Proposition \ref{irredcriterion} to the Type II components, we can split $\alpha$ relative to the boundary divisor $\Delta_{\rm II}$. The normal bundle of each component is anti-ample on the fibers of the contraction map to $X_\Gamma^{\rm BB}$, so it has non-trivial Euler class, as required.

Next, we apply Proposition \ref{splitcriterion} to the pair $(X \setminus \Delta_{\rm II} \,,\,  \Delta_{\rm III}\setminus \Delta_{\rm II})$. The map $H_1(N^*_\Delta) \to H_1(N_\Delta)$ is an isomorphism using Proposition \ref{groupcoho}.
\end{proof}

We have represented every class in $H_2(X_\Gamma^\Sigma,\Q)$ as a linear combination of classes supported on either $X_\Gamma$ or $\Delta$. We require a further strengthening of this:

\begin{theorem}\label{t:algsplitting}
Every homology class $\alpha\in H_2(X^\Sigma_\Gamma,\Q)$ can be represented as the sum of a cycle supported on the interior $X_\Gamma$ and an algebraic cycle contracted by the map to $X_\Gamma^{\rm BB}$.
\end{theorem}
\begin{proof}
By Theorem \ref{t:splitting}, we may express $\alpha$ as the sum of a class in $H_2(X_\Gamma,\Q)$ and a class supported in the boundary $\Delta$. We will treat Type II and III boundary divisors separately. Recall from \Cref{subsec-ii} that $X_\Gamma^\Sigma\to X_\Gamma^{\rm BB}$ induces a fibration $\pi_J\colon \Delta_J\to B_J$ (in finite quotients of abelian varieties) of each Type II toroidal boundary divisor over the Type II Baily-Borel modular curve, and it contracts each Type III boundary divisor $\Delta_I$ to a point.

In Type II, suppose that $\alpha_J\in H_2(\Delta_J,\Q)$, and let $\mathcal{L}_J$ be the normal
bundle of the smooth divisor $\Delta_J\subset X_\Gamma^\Sigma$. For any algebraic curve $C_J$ contained in a fiber of $\pi_J$, we have $\mathcal{L}_J\cdot C_J<0$ because $C$ is contracted by the morphism to Baily-Borel. Thus, we can choose $r\in \Q$ for which 
\[\mathcal{L}_J\cdot (\alpha_J+rC_J)=0.\]

After clearing denominators, a well-known theorem of Thom says that $N(\alpha_J+rC_J)\in H_2(\Delta_J,\Z)$ can be represented by the fundamental class of an immersed (real) oriented surface, $S$. The pullback of $\mathcal L_J$ to $S$ is a degree 0 complex line bundle, so it is smoothly trivial. By flowing along a nowhere vanishing normal vector field, $S$ is homologous to a surface $S'$ in $X_\Gamma^\Sigma \setminus \Delta_J$. We can then express
\[ \alpha_J = \frac{1}{N}[S'] - r[C_J]. \]

In Type III, suppose that $\alpha_I\in H_2(\Delta_I,\Q)$. Recall that $\Delta_I$ has simple normal crossings with dual complex homeomorphic to the hyperbolic manifold $\Gamma_I \backslash \mathbb{P} C_K$. From the $E_1$-page of the Mayer-Vietoris spectral sequence associated to the covering $\mathfrak U$ of $\Delta_I$ via neighborhoods of the irreducible toric components, we obtain the following short exact sequence for $H_2(\Delta_I,\Q)$, using that $H_1$ vanishes for all toric varieties:
\[H_2(\mathfrak U_1,\Q) \to H_2(\mathfrak U_0,\Q)  \to H_2(\Delta_I,\Q) \to H_2(\Gamma_I,\Q) \to 0.  \]

Recall that the map $N^*_{\Delta_I} \to \Delta_I$ is the quotient of a $\Gamma_I$-equivariant map $V_I \to \widetilde{\Delta}_I$. Now, we appeal to the same functoriality of the Leray-Cartan spectral sequence as in Proposition \ref{groupcoho}, this time for $H_2$ on the $E_2$-page, to obtain
\[\xymatrix{
H_1(\Gamma_I, H_1(T^n,\Q)) \ar[r]\ar[d] & H_2(N^*_{\Delta_I},\Q) \ar[r]\ar[d] & H_2(\Gamma_I,\Q) \ar[r]\ar[d] & 0\\
H_0(\Gamma_I,H_2(\widetilde{\Delta}_I,\Q)) \ar[r] & H_2(\Delta_I,\Q) \ar[r] & H_2(\Gamma_I,\Q) \ar[r] & 0.
}\]
Since $\widetilde{\Delta}_I$ is a union of toric varieties, the $\Gamma_I$-coinvariants of $H_2(\widetilde{\Delta}_I,\Q)$ are precisely the algebraic curve classes in $\Delta_I$. By a diagram chase, $\alpha_I$ can be expressed as: 
$\alpha_I = T+C_I$
where $C_I$ is in the image of $H_0(\Gamma_I,H_2(\widetilde{\Delta}_I,\Q))$ and $T$ is in the image of $H_2(N^*_{\Delta_I},\Q)$, so is supported on the interior. There is a surjection $$\textstyle \bigoplus_c H_2(\Delta_{I,c},\Q)\to H_0(\Gamma_I,H_2(\widetilde{\Delta}_I,\Q))$$ and since $\Delta_{I,c}$ is a toric variety, $H_2(\Delta_{I,c},\Q)$ is generated by algebraic curve classes.
\end{proof}

\begin{proposition}\label{p:toric-splitting} Every algebraic curve class which pushes forward to zero along $X_\Gamma^\Sigma\to X_\Gamma^{\rm BB}$ is homologically equivalent to a linear combination of one-dimensional Type III toric boundary strata. \end{proposition}

\begin{proof}
This is shown in \cite[Thm.~7.18]{alexeev-engel}. Any such curve in Type II can be degenerated into the Type III locus by limiting the image point $j\in B_J$ to a cusp of the modular curve. This reduces to the Type III case. One applies the torus action on a Type III toric boundary divisor to further break any curve into one-dimensional torus orbits.
\end{proof}

\section{Mixed mock modularity of special divisors}

In this section, we assemble all the ingredients from previous sections to prove the main theorems of the paper. 

Let $\Sigma$ be an admissible fan and assume that $\Sigma$ is regular and $\Gamma$ is neat---the general case reduces to this one by \Cref{r:simplicial}. Let $I\subset L$ be an isotropic line and let $\Sigma_I$ be the $\Gamma_I$-invariant polyhedral cone decomposition of the cone $C_K^+$ inside the hyperbolic lattice $K=I^{\bot}/I$. 

To prove \Cref{t:precise-main} (which in turn implies \Cref{t:main}), it suffices, by Poincar\'e duality, to intersect the generating series of special divisors with each class $\alpha\in H_2(X_\Gamma^\Sigma,\Q)$.
By \Cref{t:algsplitting} and \Cref{p:toric-splitting}, any class $\alpha$ can be expressed as
\[\alpha=\alpha'+\sum_{\sigma}d_\sigma [C_\sigma],\quad d_\sigma\in\Q,\]
where $\alpha'\in H_2(X_\Gamma,\Q)$ and $C_{\sigma}\simeq \mathbb{P}^1$ are the Type III toric boundary strata associated to codimension $1$ cones $\sigma\in \Sigma_I$ ranging over all $I$.

Since the pairing of the generating series of special divisors with the class $\alpha'$ is a holomorphic modular form of weight $1+\frac{n}{2}$ and representation $\overline{\rho}_L$ by \cite{kudla-millson}, we can assume that $\alpha'=0$ and that $\alpha=[C]$ is the class of a single toric curve determined by a cone $\sigma$. 

By \Cref{t:intersection-formula}, the $\C[L^\vee/L]$-valued generating series of intersecting $[C]$ with special cycles is equal to \[\sum_{\lambda\in K^{\vee}}p^{+}(\lambda)q^{-Q(\lambda)}p_K^L(\mathfrak{e}_{[\lambda]}).\] Here
$p^+(\lambda)=\tfrac{1}{2}\left(|\lambda\cdot c^+|+|\lambda\cdot c^-|+\sum_i a_i|\lambda\cdot c_i|\right)$ and $\{c_i\}_{i=1}^{n-1}$ are generators of the rays of $\sigma$, which, with $c^+$, $c^-$, respectively generate the maximal cones $\sigma^+$, $\sigma^-$ containing $\sigma$. To make the notation compatible with \Cref{sec2}, we declare $c_n:=c^+$, $c_{n+1}:=c^-$, and $a_n=a_{n+1}=1$ so that $\{c_i\}_{i=1}^{n+1}$ are a collection of integral vectors in $C_K^+$ for which $\sum a_ic_i=0$. We have $\Delta_{I, c_i}\cdot C=a_i$ for all $i=1,\dots,n+1$. This holds even when $c_i$ is isotropic, in which case $\Delta_{J_i}\cdot C=a_i$ for the corresponding rank $2$ isotropic lattice $J=I\oplus \Z c_i$.

For any non-isotropic ray $\R_{\geq 0}c$ in $\Sigma_I$, let $N_c=Q(c)>0$ and let $c^{\bot}$ be the orthogonal complement of $c$ in $K$. It is a negative definite lattice  of rank $n-1$. Its vector-valued theta series
\[\Theta_{c^{\bot}}=\sum_{\delta \in (c^{\bot})^\vee}q^{-Q(\delta)}\mathfrak{e}_{[\delta]}\in\C[[q^{\frac{1}{D_{c^\perp}}}]]\otimes \C[(c^\perp)^\vee/c^\perp]\]
is a holomorphic modular form of weight $\frac{n-1}{2}$ with respect to the Weil representation $\overline{\rho}_{c^{\bot}}$. Similarly, for any isotropic ray $\R_{\geq 0}c$, the vector-valued theta series $$\Theta_M = \sum_{\gamma\in M^\vee} q^{-Q(\gamma)}\mathfrak{e}_{[\gamma]}\in \C[[q^{\frac{1}{D_M}}]]\otimes \C[M^\vee/M]$$ is a holomorphic modular form of weight $\tfrac{n}{2}-1$ with respect to $\overline{\rho}_M$, where $M=c^\perp/c$.
By \Cref{typeII+typeIII}, the function $\Theta_{\sigma}$ defined by 
\begin{align*} \sum_{\lambda\in K^{\vee}}p^{+}(\lambda)  q^{-Q(\lambda)}\mathfrak{e}_{[\lambda]}-\!\!\sum_{Q(c_i)>0}\! \frac{a_i}{2} \uparrow_{c_i^\bot\oplus \Z c_i}^K(\Theta_{c_i^{\bot}}\otimes F_{N_i}^+)+\!\sum_{Q(c_i)=0} \frac{a_i}{12} E_2\uparrow_{M_i}^{K}(\Theta_{M_i})\end{align*}
is a (weakly) holomorphic modular form of weight $1+\frac{n}{2}$ and representation $\overline{\rho}_K$. On the other hand, 
\begin{align*} \bigg[[Z&(0,0)]\otimes \mathfrak{e}_0 + \sum_{\beta\in L^\vee/L}\sum_{m\in -Q(\beta)+\Z} [Z(\beta,m)]\otimes \mathfrak{e}_\beta q^m\,- \\ &\sum_{(I,c)} \frac{1}{2} \uparrow_{c^\bot\oplus \Z c}^{L}(\Theta_{c^{\bot}}\otimes F_{N}^+)\otimes \Delta_{I,c} + \sum_J \frac{1}{12} E_2p_M^L(\Theta_M)\otimes \Delta_J\bigg]\cdot C = \\
& \hspace{-12pt}\sum_{\lambda\in K^{\vee}}p^{+}(\lambda)q^{-Q(\lambda)}p_K^L(\mathfrak{e}_{[\lambda]}) \,-\\ &\sum_{Q(c_i)>0}\!\frac{a_i}{2}\uparrow_{c_i^\bot\oplus \Z c_i}^{L}(\Theta_{c_i^{\bot}}\otimes F_{N_i}^+)+\!\sum_{Q(c_i)=0} \frac{a_i}{12} E_2\uparrow_{M_i}^L(\Theta_{M_i})= p_K^L(\Theta_\sigma)\end{align*}
is (weakly) holomorphic modular form of weight $1+\frac{n}{2}$ and representation $\overline{\rho}_L$. Here $i$ ranges only over $1,\dots,n+1$ because all other components of the boundary $\Delta$ intersect $C$ to be zero. 
Hence the pairing of the series in \Cref{t:precise-main} with $[C]$ is a (weakly) holomorphic modular form, which concludes the proof.

\bibliographystyle{alpha}
\bibliography{bibliographie}
\end{document}